\documentclass[12pt]{amsart} 
\usepackage{amssymb,amsmath} 
\usepackage[dvipdfmx]{graphicx} 
\usepackage{color} 
\usepackage{comment} 
\usepackage{amsthm} 
\newtheorem{defi}{Definition}[section]
\newtheorem{thm}[defi]{Theorem} 
\newtheorem{cor}[defi]{Corollary} 
\newtheorem{lem}[defi]{Lemma} 
\newtheorem{rk}[defi]{Remark} 
\makeatletter
 
\@addtoreset{equation}{section} 
\makeatother
\usepackage{framed}

\allowdisplaybreaks[1]

\begin{document}

\title[Area preserving curvature flow]
{Global stability of traveling waves \\
for an area preserving curvature flow \\
with contact angle condition}

\author{Takashi Kagaya}
\address{Institute of Mathematics for Industry, Kyushu University, Japan}
\email{kagaya@imi.kyushu-u.ac.jp}

\medskip


\thanks{
The author is partially supported by JSPS KAKENHI Grant Numbers JP19K14572, JP18H03670 and JST-Mirai Program Grant Number JPMJMI18A2.  
}

\begin{abstract}
We consider an evolving plane curve with two endpoints that can move freely on the $x$-axis with generating constant contact angles. 
We discuss the asymptotic behavior of global-in-time solutions when the evolution of this plane curve is governed by area-preserving curvature flow equation. 
The main result shows that any moving curve converges to a traveling wave if the moving curve starts from an embedded convex curve and remains bounded in global time. 
\end{abstract}

\maketitle \setlength{\baselineskip}{18pt}

\section{Introduction}\label{sec:int}

The purpose of this paper is to investigate the geometrical behavior of the area-preserving curvature flow of a planar curve having two endpoints on the $x$-axis with fixed interior contact angles to this axis. 
The problem is formulated as follows. 
Let $X(p,t) = (x(p,t),y(p,t)): [-1,1] \times [0,T) \to \mathbb{R}^2$ represents the position of the curve $\gamma(t)$ at time $t$. 
Denote by $s$ the arc-length parameter along $\gamma(t)$ measured from $X(-1,t)$ to $X(1,t)$. 
The unit normal vector $\mathcal{N}$ and the signed curvature $\kappa$ are defined by 
\[ \mathcal{N} := \left( -\dfrac{\partial y}{\partial s}, \dfrac{\partial x}{\partial y} \right), \quad \kappa := \left\langle \left( \dfrac{\partial^2 x}{\partial s^2}, \dfrac{\partial^2 y}{\partial s^2}\right), \mathcal{N} \right \rangle, \]
where $\langle \cdot, \cdot \rangle$ is the Euclidean inner product defined on $\mathbb{R}^2$. 
The motion is governed by 
\begin{align}
&\left\langle\dfrac{\partial X}{\partial t}, \mathcal{N}\right\rangle = \kappa - \dfrac{\int_{\gamma(t)} \kappa \; ds}{L(t)}, \label{apcf}\\
&y(\pm1, t) = 0, \label{bc1}\\
&\mathcal{N}(-1,t) = (-\sin \psi_-, \cos \psi_-), \quad \mathcal{N}(1,t) = (\sin \psi_+, \cos \psi_+), \label{bc2}
\end{align}
where $L(t)$ is the length of $\gamma(t)$ and $\psi_\pm \in (0,\pi)$ are constants. 
We assume that 
\begin{equation}
X(\cdot,0) \; \mbox{is injective and} \; x(-1,0) < x(1,0) \tag{A1} \label{as1}
\end{equation}
In this case, the normal vector $\mathcal{N}$ is outward pointing at $t=0$. 
The angles $\psi_-$ and $\psi_+$ are interior contact angles at the endpoints $X(-1,t)$ and $X(1,t)$, respectively (see Figure \ref{fig:angle}). 
Our aim is to study the asymptotic behavior of the flow as $t \to \infty$. 

\begin{figure}[t]
\begin{center}
\scalebox{0.35}{\includegraphics{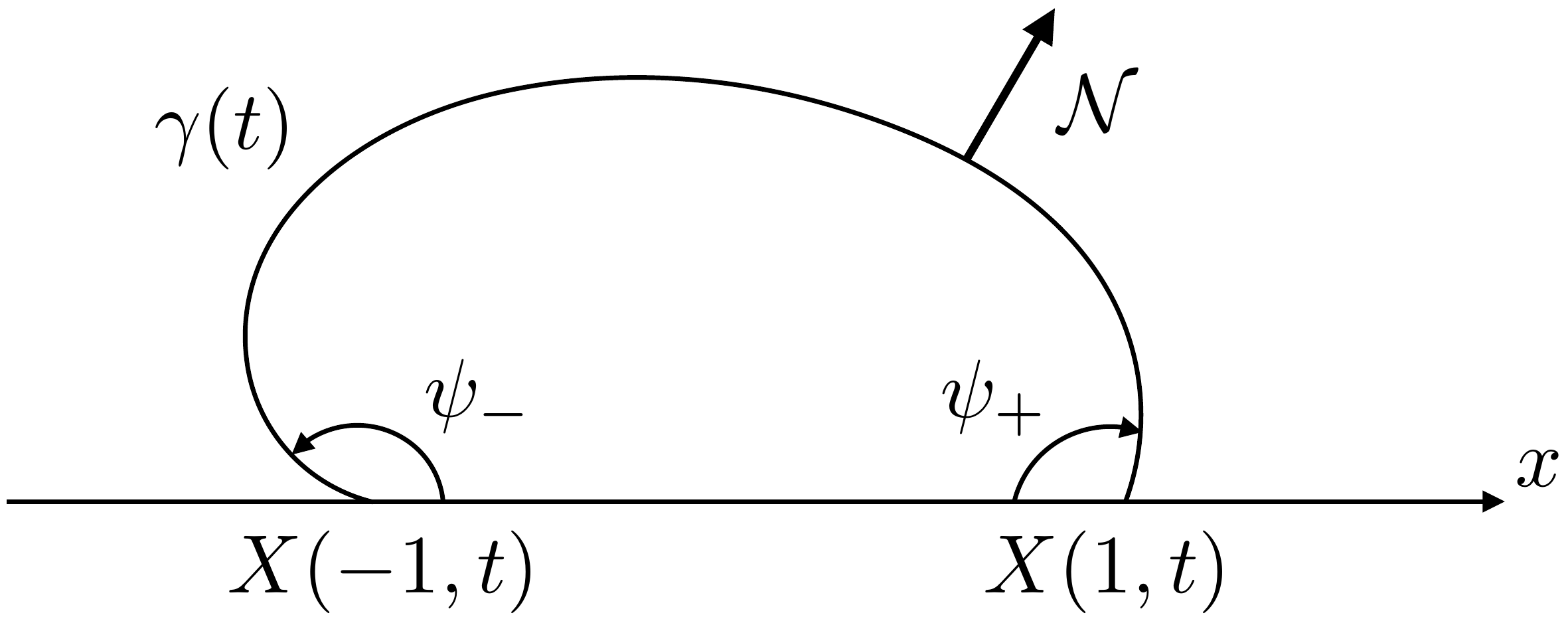}}
\end{center}
\caption{Direction of the contact angles and the unit normal vector.}
\label{fig:angle}
\end{figure}

Before stating our main results and motivation in detail, we discuss some known results associated with our problem. 
The family of Jordan curves governed by \eqref{apcf} was introduced by Gage \cite{G} as the $L^2$-gradient flow of the length of a Jordan curve under the area-preserving variations. 
Therefore, the length of $\gamma(t)$ is non-increasing in time, while the area enclosed by $\gamma(t)$ is preserved for the area-preserving curvature flow $\{\gamma(t)\}$. 
This variational structure indicates that if $\gamma(t)$ exists globally in time, then it converges to a critical point for length under the area constraint, that is, a circle. 
Indeed, he proved this fact for convex initial curves through the global-in-time existence of the flow. 
The local exponential stability of the circles without the assumption of the convexity on the initial curves was proved in \cite{ES} and the exponential convergence of arbitrary global-in-time flow to a circle was proved in \cite{NN}. 

On the other hand, the free boundary problem \eqref{apcf}--\eqref{bc2} can be introduced as a formal $L^2$-gradient flow of 
\begin{equation}\label{def-ene} 
E(\gamma) := L(\gamma) - x_+(\gamma) \cos \psi_+ + x_-(\gamma) \cos \psi_- 
\end{equation}
under the constraint so that the (signed) area 
\begin{equation}\label{def-area2} 
A_\gamma := \dfrac{1}{2} \int_{\gamma} \langle X_\gamma, \mathcal{N}_\gamma \rangle \; ds 
\end{equation}
is preserved, where $\gamma$ is a simple curve having two endpoints on the $x$-axis, $L(\gamma)$ is its length, $x_+(\gamma)$ (resp.\ $x_-(\gamma)$) is the $x$-coordinate of the right (resp.\ left) endpoint of $\gamma$, $X_\gamma$ is the position of $\gamma$ and $\mathcal{N}_\gamma$ is the outward pointing unit normal vector. 
We note that $A_\gamma$ coincide with the area enclosed by $\gamma$ and the $x$-axis if $\gamma$ is in the upper half space (see \cite{Z}, for example). 
Therefore, we may expect that any global-in-time solution of \eqref{apcf}--\eqref{bc2} converges to a critical point of $E$ under the area constraint. 
However, if $\psi_+ \neq \psi_-$, then the functional $E$ is not bounded from below and there is no critical point (see also Remark \ref{rk:ene}). 
Thus, the asymptotic behavior of global solutions to \eqref{apcf}--\eqref{bc2} is not obvious in view of the variational structure. 
Under the assumption that 
\begin{itemize}
\item[(A')] $\gamma(0)$ is $C^2$ and a concave graph (i.e.\ $\kappa < 0$) over the $x$-axis, and satisfies \eqref{bc1}, \eqref{bc2} for $t=0$, and \eqref{as1}, 
\end{itemize}
the free boundary problem \eqref{apcf}--\eqref{bc2} was studied by Shimojo and the author \cite{SK}. 
In particular, they proved the following statements. 
\begin{itemize}
\item The curve $\gamma(t)$ remains uniformly bounded in global time, and $\gamma(t)$ is a concave graph at any time. 
\item There exists a traveling wave such that the profile curve $\mathcal{W}(0)$ is a concave graph and it is unique up to the scaling and the translation parallel to the $x$-axis. 
Here, the traveling wave is defined by a classical global-in-time solution to \eqref{apcf}--\eqref{bc2} for which the curve $\mathcal{W}(t)$ satisfies 
\begin{equation}\label{def-tw}
\mathcal{W}(t) = \mathcal{W}(0) + ct\vec{e}_1,
\end{equation}
where $\vec{e}_1 = (1,0)$ and $c \in \mathbb{R}$ is some constant. 
\item Any traveling wave is locally exponentially stable. 
That is, if the curvature of $\gamma(0)$ is sufficiently ``close'' to the curvature of $\mathcal{W}(0)$ and if the area enclosed by $\gamma(0)$ and the $x$-axis is equal to the area enclosed by $\mathcal{W}(0)$ and the $x$-axis, $\gamma(t)$ converges exponentially to $\mathcal{W}(t) + a \vec{e}_1$ under the Hausdorff metric for some $a \in \mathbb{R}$. 
\end{itemize}
We note that they only study the case $\psi_\pm \in (0,\pi/2)$ so that the initial curve $\gamma(0)$ and the profile curve $\mathcal{W}(0)$ can be assumed to represent concave graphs. 
For the case of more general contact angles, the author and Kohsaka \cite{KK} studied the existence of traveling waves and geometric properties of these traveling waves under an  assumption associated to the winding number $-\int_{\mathcal{W}(0)} \kappa_{\mathcal{W}} \; ds = \psi_+ + \psi_-$, where $\kappa_{\mathcal{W}}$ is the curvature of the profile curve $\mathcal{W}(0)$. 
They proved the existence of a traveling wave for any contact angles $\psi_\pm \in (0,\pi)$. 
Furthermore, the traveling wave is concave and unique up to scaling and translation (see Theorem \ref{thm:ex-traveling} for details). 

Based on the known results, we are interested in the asymptotic behavior of a solution to our problem without the assumptions that the initial curve is a graph or the ``closeness'' between the initial curve and a traveling wave. 
The first result of the present paper is local-in-time existence and uniqueness of a solution to \eqref{apcf}--\eqref{bc2}. 
We assupme \eqref{as1} and 
\begin{align}
& X(\cdot,0) \in C^2([-1,1]) \; \mbox{and satisfies \eqref{bc1} and \eqref{bc2} for} \; t=0, \tag{A2} \label{as2}\\
& -\int_{\gamma(0)} \kappa(\cdot,0) \; ds = \psi_+ + \psi_-, \tag{A3} \label{as3}
\end{align}
and find a unique solution $X$ satisfying 
\begin{equation}\label{restrict-para}
\left|\dfrac{\partial X(p,t)}{\partial p}\right| = \dfrac{L(t)}{2} \quad \mbox{for} \quad p \in [-1,1]. 
\end{equation}
Since the solution can be re-parametrized with respect to $p \in (-1,1)$, the solution $X$ to \eqref{apcf}--\eqref{bc2} is not unique without the restriction \eqref{restrict-para}, while the time evolution of the shape $\gamma(t)$ is unique. 

\begin{thm}\label{thm:main1}
Assume \eqref{as1}--\eqref{as3}. 
Let $\alpha \in (1/2,1)$ be an arbitrary constant. 
Then, there exists a time $T>0$ such that a unique smooth solution 
\begin{equation}\label{reg-sol} 
X \in C^{{1+\alpha},(1+\alpha)/2}([-1,1] \times [0,T)) \cap C^\infty([-1,1]\times(0,T)) 
\end{equation}
to \eqref{apcf}--\eqref{bc2} satisfying \eqref{restrict-para} exists. 
\end{thm}

We note that the argument for Theorem \ref{thm:main1} is based on analytic semigroup theory as in \cite{SK}. 
Since the flow is unique, we can discuss the stability of the traveling waves $\mathcal{W}(t)$ obtained by \cite{KK} by studying the asymptotic behavior of global-in-time solutions to \eqref{apcf}--\eqref{bc2}. 
In the present paper, we also assume the following properties to study the asymptotic behavior of solutions: 
\begin{itemize}
\item[(A4)] The initial curve is concave, i.e., $\kappa(p,0) < 0$ for $p \in [-1,1]$. 
\item[(A5)] The solution obtained in Theorem \ref{thm:main1} exists in global time. 
\item[(A6)] There exists a constant $c_1 >0$ such that $\sup_{(p,t) \in (-1,1) \times (0,\infty)} |\kappa(p,t)| \le c_1$.
\item[(A7)] There exists a constant $c_2 > 0$ such that $\sup_{t \in (0,\infty)} L(t) \le c_2$. 
\end{itemize}
We denote by $A(t)$ and $A_{\mathcal{W}}$ the (signed) areas
\begin{align}
&A(t) := \dfrac{1}{2} \int_{\gamma(t)} \langle X, \mathcal{N} \rangle \; ds, \label{def-area} \\
&A_{\mathcal{W}} := \dfrac{1}{2} \int_{\mathcal{W}(0)} \langle X_{\mathcal{W}}, \mathcal{N}_{\mathcal{W}} \rangle \; ds, \label{def-area3}
\end{align}
where $X_{\mathcal{W}}$ and $\mathcal{N}_{\mathcal{W}}$ are the position and the outward pointing unit normal vector of the profile curve $\mathcal{W}(0)$, respectively. 
The area $A(t)$ is preserved with respect to time $t$ (see Lemma \ref{lem:area-pre}). 
Because the concavity of $\gamma(t)$ is preserved (see Lemma \ref{lem:concave}), $A(t)$ coincides with the area enclosed by $\gamma(t)$ and the $x$-axis if the curve $\gamma(t)$ is simple. 
We note that we do not know whether $\gamma(t)$ is simple or not for $t>0$ even if we assume the simplicity of the initial curve $\gamma(0)$ as in \eqref{as1} (see also Remark \ref{rk:pre-sim}). 
The possibility of a loss of embeddedness makes it difficult to prove the global boundedness in (A6) and (A7) in general (see also Remark \ref{rk:regularity}). 
The main result is stated as follows. 

\begin{thm}\label{thm:main2}
Let $\gamma(t)$ be the family of curve obtained in Theorem \ref{thm:main1}. 
Assume {\rm (A4)}--{\rm (A7)}. 
Let $\mathcal{W}(t)$ be the traveling wave obtained by \cite{KK} such that $A_{\mathcal{W}} = A(0)$. 
Denote by $\kappa(\cdot,t)$ and $\kappa^*$ the curvatures of $\gamma(t)$ and $\mathcal{W}(0)$, respectively. 
Then, $\kappa(\cdot,t)$ converges exponentially to $\kappa^*$ in $C^\infty$, and there exists a constant $a \in \mathbb{R}$ such that $\gamma(t)$ converges to $\mathcal{W}(t) + a\vec{e}_1$ under the Hausdorff metric, where $\vec{e}_1 = (1,0)$. 
\end{thm}

As we mentioned, Shimojo and the author \cite{SK} proved the uniformly boundedness of the solution under the assumption (A'). 
That is, (A5)--(A7) hold only if we assume (A'). 
Therefore the following corollary follows immediately from Theorem \ref{thm:main2}. 

\begin{cor}
Let $\gamma(0)$ satisfies {\rm (A')}. 
Denote by $\gamma(t)$ and $\mathcal{W}(t)$ the global-in-time solution to \eqref{apcf}--\eqref{bc2} obtained by \cite{SK} and the traveling wave obtained by \cite{KK, SK} such that $A_{\mathcal{W}} = A(0)$, respectively. 
Then, $\kappa(\cdot,t)$ converges exponentially to $\kappa^*$ in $C^\infty$ and there exists a constant $a \in \mathbb{R}$ such that $\gamma(t)$ converges to $\mathcal{W}(t) + a\vec{e}_1$ under the Hausdorff metric, where $\kappa(\cdot,t)$ and $\kappa^*$ are the curvature of $\gamma(t)$ and $\mathcal{W}(0)$, respectively.
\end{cor} 

We now recall some results related to our problem. 
As previously mentioned, Gage \cite{G} proved that any flow starting from a convex Jordan curve and governed by \eqref{apcf} converges to a circle. 
Chao, Ling and Wang \cite{CLW} gave a detailed proof of the result by Gage \cite{G} by analyzing the the uniformly boundedness of the curvature. 
We note that their argument for showing the uniformly boundedness of the curvature seems to be difficult to modify for our problem. 
For the motion of a graph governed by $\langle \frac{\partial X}{\partial t}, \mathcal{N} \rangle = \kappa + 1$, \eqref{bc1} and \eqref{bc2}, Guo, Matano, Shimojo and Wu \cite{GMSW} studied the asymptotic behavior of solution. 
They proved the global-in-time existence of a unique solution for any initial graph (without a concavity assumption on the initial graph) and gave a classification of the asymptotic behavior of the solution, namely, any solution either (A) expands far away as $t \to \infty$, (B) remains bounded and converges to a traveling wave as $t \to \infty$, or (C) shrinks to a point in finite time. 
In particular, the convergence in the case (B) was proved without any ``closeness'' assumption between the initial graph and a traveling wave. 
However, it is difficult to apply their arguments to our problem because the arguments are based on the zero-number principle as in \cite{An}. 
In general, the zero-number principle does not hold for differential equations such as \eqref{apcf} that have a nonlocal term. 

We also remark that various geometric flows with contact angle conditions similar to ours have been studied. 
For example, we refer the reader to \cite{BN, CGK, CG, many} for the curve-shortening flows, \cite{EF} for the area-preserving curvature flows and \cite{AB, KK2} for the surface diffusion. 

Area-preserving curvature flows can also be studies for general closed curves. 
Gage \cite{G} also suggested an example that should exhibit a loss of embeddedness and a blow-up of the curvature for a flow starting from a non-convex embedded curve. 
These singularities were confirmed numerically by Mayer \cite{M}. 
The motion of a non-simple closed curve was studied by \cite{EI, WK, WWY} and, in particular, threshold properties between blow-up and global-in-time existence of the flow were discussed. 
Escher and Ito \cite{EI} proved that the curvature blows up in finite time if the rotation index of the initial curve is $1$ and the signed area defined by the integration of the support function as in \eqref{def-area} is negative. 
A criterion similar to that above appears in area-preserving curvature flow with free boundary for convex closed curves \cite{Ma}, and also possibly appears in our problem (see Remark \ref{rk:posi-area-TW}). 

The present paper is organized as follows. 
In Section \ref{sec:short-ex}, we prove the existence of a unique short time solution as in Theorem \ref{thm:main1}. 
In Section \ref{sec:pre}, we study the uniform negativity of the curvature. 
Some known results are also stated in details to discuss the stability of the traveling waves. 
The known results relate to the existence theory for traveling waves in \cite{KK} and the local stability theory of the traveling waves in \cite{SK}. 
In Section \ref{sec:main}, we construct a Lyapunov functional of the curvature and analyze the $\omega$-limit points of the curvature. 
Because the Lyapunov functional does not depend on the position of a curve, it follows that we will find a different energy structure from the variational structure following from \eqref{def-ene}. 
We also prove Theorem \ref{thm:main2} in Section \ref{sec:main}. 


\section{Short time existence}\label{sec:short-ex}

In \cite{SK}, the short time existence of a flow governed by \eqref{apcf}--\eqref{bc2} starting from a concave graph $\gamma(0)$ was studied. 
Our goal is to extend their argument to apply to our problem without the assumption that the initial curve $\gamma(0)$ is a graph. 
For the reader's convenience, we summarize the argument from \cite{SK}. 
They first rewrite the evolution equation using the length $L(t)$ and the angle $\Theta$ between the tangent vector of a solution $X$ and the $x$-axis from the evolution equation using the height function $y(x,t)$ of $\gamma(t)$.
If we apply a scaling, then the problem can be converted into a simple semi-linear problem, which makes it possible to apply the well-documented general theory for semi-linear problems. 
We apply the same argument below. 
We also note that the original argument comes from \cite{GMSW}. 

Let $\mathcal{T}$ be the tangent vector for a solution $X$, that is, 
\[ \mathcal{T} := \left( \dfrac{\partial x}{\partial s}, \dfrac{\partial y}{\partial s} \right). \]
Since the motion is governed by \eqref{apcf}, there exists some function $\alpha: (-1,1) \times (0,T)$ such that 
\begin{equation}\label{tau-eq}
\dfrac{\partial X}{\partial t} = \left(\kappa - \dfrac{\int_{\gamma(t)} \kappa \; ds}{L(t)} \right) \mathcal{N} + \alpha \mathcal{T}. 
\end{equation}
The angle function $\Theta$ is defined by 
\begin{equation}\label{def-angle-fun}
\mathcal{T}(p,t) = (\cos \Theta(p,t), \sin \Theta(p,t)) \quad \mbox{for} \quad (p,t) \in [-1,1] \times [0,T),
\end{equation}
Denote by $\upsilon (p,t)$ the length element of $\gamma(t)$, that is, 
\[
\upsilon(p,t) := \left|\dfrac{\partial F}{\partial p}(p,t) \right| \quad \mbox{for} \quad (p,t) \in [-1,1] \times [0,T). 
\]
Therefore, the following equality holds: 
\begin{equation}\label{deri-p-s} 
\dfrac{\partial}{\partial s} = \dfrac{1}{\upsilon} \dfrac{\partial}{\partial p}. 
\end{equation}
We also note that the relation 
\begin{equation}\label{deri-theta}
\dfrac{\partial \Theta}{\partial s} = \dfrac{1}{\upsilon} \dfrac{\partial \Theta}{\partial p} = \kappa
\end{equation}
holds. 
The following lemma will be applied to introduce a semi-linear problem and to analyze some geometric properties of $\gamma(t)$. 
We remark that we do not assume \eqref{restrict-para} in this lemma. 

\begin{lem}
Assume a smooth solution \eqref{reg-sol} to \eqref{apcf}--\eqref{bc2} with the initial assumptions \eqref{as1}--\eqref{as3} exists in some time interval $[0,T)$. 
Then, the curvature $\kappa$ of $\gamma(t)$ satisfies 
\begin{equation}
\kappa_t = \kappa_{ss} + \alpha \kappa_s + \kappa^2 \left(\kappa + \dfrac{\psi_+ + \psi_-}{L(t)} \right) \quad {\rm for} \quad -1 \le p \le 1, \; \; 0 < t < T, \label{k-s-eq}
\end{equation}
Furthermore, $\alpha$, $\upsilon$ and $\Theta$ satisfy 
\begin{align}
&\; \alpha(\pm1,t) = \pm \cot \psi_\pm \left(\kappa(\pm1,t) + \dfrac{\psi_+ + \psi_-}{L(t)}\right) \quad {\rm for} \quad 0 < t < T, \label{bdry-ti-V}\\
&\; \dfrac{\partial \upsilon}{\partial t} = \left\{ - \left(\kappa + \dfrac{\psi_+ + \psi_-}{L(t)} \right)\kappa + \alpha_s \right\} \upsilon \quad {\rm for} \quad -1 \le p \le 1, \; \; 0 < t <T, \label{deri-leng} \\
&\; \dfrac{\partial \Theta}{\partial t} = \kappa_s + \alpha \kappa \quad \mbox{for} \quad -1 \le p \le 1, \; \; 0<t<T. \label{eq-Theta}
\end{align}
\end{lem}

\begin{proof}
By virtue of the assumption \eqref{as3}, $-\int_{\gamma(t)} \kappa \; ds = \psi_+ + \psi_-$ holds for $t \in [0, T)$. 
Therefore, \eqref{tau-eq} can be re-formulated as 
\begin{equation}\label{tau-eq2} 
\dfrac{\partial X}{\partial t} = \left(\kappa + \dfrac{\psi_+ + \psi_-}{L(t)} \right) \mathcal{N} + \alpha \mathcal{T}. 
\end{equation}
For the equalities \eqref{k-s-eq}, \eqref{deri-leng} and \eqref{eq-Theta}, we refer to \cite[Section 1.3]{CZ}. 
We prove only \eqref{bdry-ti-V}. 
By differentiating \eqref{bc1} with respect to $t$, we see that $\frac{\partial y}{\partial t} (\pm1, t) = 0$. 
Furthermore, by virtue of \eqref{bc2} and the choice of the direction of $\mathcal{N}$, we may see that the $y$-coordinate of $\mathcal{N}$ and $\mathcal{T}$ are $\cos \psi_\pm$ and $\mp\sin \psi_\pm$ at $p=\pm1$, respectively. 
Therefore, the equality \eqref{bdry-ti-V} follows from the $y$-coordinate of \eqref{tau-eq2}. 
\end{proof}

Using \eqref{bdry-ti-V} and \eqref{deri-leng}, we can calculate the derivative of the length $L(t)$ and the $x$-coordinate of the left endpoint $x(-1,t)$. 

\begin{lem}
Assume a smooth solution \eqref{reg-sol} to \eqref{apcf}--\eqref{bc2} with the initial assumptions \eqref{as1}--\eqref{as3} exists in some time interval $[0,T)$. 
Then, the length $L(t)$ of $\gamma(t)$ and the $x$-coordinate of the left endpoint $x(-1,t)$ satisfy 
\begin{align}\label{deri-ene}
&\begin{aligned}
\dfrac{d}{dt}L(t) =& \; \cot \psi_+ \left(\kappa(1,t) + \dfrac{\psi_+ + \psi_-}{L(t)}\right) + \cot \psi_- \left(\kappa(-1,t) + \dfrac{\psi_+ + \psi_-}{L(t)}\right) \\
&\; + \dfrac{(\psi_++\psi_-)^2}{L(t)} - \int_{\gamma(t)} \kappa^2 \; ds, 
\end{aligned} \\
&\dfrac{\partial x}{\partial t}(-1,t) = - \dfrac{1}{\sin\psi_-} \left(\kappa(-1,t) + \dfrac{\psi_+ + \psi_-}{L(t)} \right). \label{deri-x} 
\end{align}
\end{lem}

\begin{proof}
First, we prove \eqref{deri-ene}. 
By virtue of \eqref{deri-leng}, we have 
\[ \dfrac{d}{dt}L(t) = \dfrac{d}{dt} \int_{-1}^1 \upsilon \; dp = \int_{-1}^1 \dfrac{\partial \upsilon}{\partial t} \; dp = \int_{\gamma(t)} - \left(\kappa + \dfrac{\psi_+ + \psi_-}{L(t)} \right)\kappa + \alpha_s \; ds. \]
Since $-\int_{\gamma(t)} \kappa \; ds = \psi_+ + \psi_-$ follows from the assumption \eqref{as3}, we obtain \eqref{deri-ene} by substituting \eqref{bdry-ti-V} into the above equality. 

The equality \eqref{deri-x} follows from \eqref{bdry-ti-V} and the $x$-coordinate of \eqref{tau-eq2} since the $x$-coordinate of $\mathcal{N}$ and $\mathcal{T}$ are $-\sin\psi_-$ and $\cos\psi_-$ at $p=-1$, respectively. 
\end{proof}

We now use a parameter $p$ satisfying \eqref{restrict-para} to determine $\alpha$. 

\begin{lem}
Assume a smooth solution \eqref{reg-sol} to \eqref{apcf}--\eqref{bc2} with the initial assumptions \eqref{as1}--\eqref{as3} exists in some time interval $[0,T)$. 
Let the parameter $p$ of $X$ be chosen to satisfy \eqref{restrict-para}. 
Then, $\alpha$ in \eqref{tau-eq} is given by 
\begin{equation}\label{alpha-p}
\alpha(p,t) = \int_{-1}^p \dfrac{2}{L(t) \kappa} \left(\kappa + \dfrac{\psi_++\psi_-}{L(t)} \right) + \dfrac{L'(t)}{2} \; dp - \cot \psi_- \left(\kappa(-1,t) + \dfrac{\psi_+ + \psi_-}{L(t)} \right) 
\end{equation}
for $(p,t) \in [-1,1] \times (0,T)$, where $L'(t)$ is the derivative of $L(t)$ with respect to $t$. 
\end{lem}

\begin{proof}
Let $\langle \cdot, \cdot \rangle$ be the Euclidean inner product defined on $\mathbb{R}^2$ and $V$ denote by the normal velocity $\kappa + (\psi_+ + \psi_-)/L(t)$ for simplicity. 
Differentiating square of the both side of \eqref{restrict-para}, we obtain 
\begin{equation*}
\begin{aligned}
\dfrac{L(t)L'(t)}{2} =&\; 2\langle X_{pt}, X_p \rangle = 2 \langle (V\mathcal{N} + \alpha \mathcal{T})_p, X_p \rangle \\
=&\; \dfrac{(L(t))^2}{2}\langle (V\mathcal{N} + \alpha \mathcal{T})_s, \mathcal{T} \rangle = \dfrac{(L(t))^2}{2} (\alpha_s - \kappa V). 
\end{aligned}
\end{equation*}
Here, the relation \eqref{deri-p-s} and the Frenet-Serret formula $\mathcal{N}_s = - \kappa \mathcal{T}$ have used. 
This implies 
\begin{equation}\label{alpha-p-1}
\alpha_p = \dfrac{L(t) \kappa V + L'(t)}{2}. 
\end{equation}
Integrating \eqref{alpha-p-1} on the interval $[-1,p]$ and substituting \eqref{bdry-ti-V}, we have \eqref{alpha-p}. 
\end{proof}

Now use the transformations 
\begin{equation}\label{trans} 
z := \dfrac{s+1}{2}, \quad \tau := \int_0^t \dfrac{d\tilde{t}}{(L(\tilde{t}))^2} 
\end{equation}
and let $v(z,\tau) := \Theta(p,t), \eta(\tau) := \log L(t)$. 
Then, by \eqref{deri-theta}, \eqref{eq-Theta}, \eqref{deri-ene} and \eqref{alpha-p}, we see that $v$ and $\eta$ satisfy 
\begin{equation}\label{normal-eq}
\begin{cases}
v_\tau = v_{zz} + (P(z,\tau) + Q(\tau) z) v_z & \mbox{for} \quad 0 < z < 1, \; \; \tau > 0, \\
\eta'(\tau) = Q(\tau) & \mbox{for} \quad \tau >0, \\
v(0,\tau) = \psi_-, \quad v(1,\tau) = -\psi_- & \mbox{for} \quad \tau>0, 
\end{cases}
\end{equation}
where 
\begin{align*}
& P(z,\tau) = (\psi_+ + \psi_-) (v - \psi_- - \cos \psi_-) - v_z(0,\tau) \cot\psi_- + \int_0^z v_z^2 \; dz, \\
& Q(\tau) = \cot\psi_+ (v_z(1,\tau) + \psi_+ + \psi_-) + \cot \psi_- (v_z(0,\tau) + \psi_+ + \psi_-) + (\psi_+ + \psi_-)^2 - \int_0^1 v_z^2 \; dz. 
\end{align*}
We have applied \eqref{bc2}, a choice of the direction for $\mathcal{N}$, and the definition of $\Theta$ \eqref{def-angle-fun} to introduce the boundary condition of \eqref{normal-eq}. 
If we obtain a solution $(v,\eta)$ to \eqref{normal-eq}, then we can construct an angle function $\Theta$ and a length $L(t)$ by inverting the transformations \eqref{trans}. 
Hence a solution to \eqref{apcf}--\eqref{bc2} can be constructed from $\Theta$, $L(t)$ and the left endpoint $x(-1,0)$ as 
\begin{equation}\label{parametrization} \begin{aligned}
&x(p,t) = x(-1,0) -\dfrac{1}{\sin \psi_-} \int_0^t \dfrac{2}{L(\tilde{t})}\Theta_p(-1,\tilde{t}) + \dfrac{\psi_+ + \psi_-}{L(\tilde{t})} \; d\tilde{t} + \int_{-1}^p \dfrac{L(t)}{2} \cos\Theta(\tilde{p},t) \; d\tilde{p}, \\
&y(p,t) = \int_{-1}^p \dfrac{L(t)}{2} \sin \Theta(\tilde{p}, t) \; d\tilde{p}. 
\end{aligned} \end{equation}
This formulation is obtained by integrating \eqref{def-angle-fun} and \eqref{deri-x} with respect to $s$ and $t$, respectively. 
Thus, it is enough to study the short time existence of solutions to \eqref{normal-eq} to prove Theorem \ref{thm:main1}. 
Equation \eqref{normal-eq} has already been studied in \cite{SK}, where analytic semigroup theory was applied  to obtain a local-in-time solution and show its uniqueness. 
In the present paper, we mention some details of the statements of existence, and refer to \cite{GMSW, SK} for the proof. 

\begin{lem}\label{short-v}
Assume $v_0 \in C^{\alpha+\varepsilon} ([0,1])$ for some $\alpha \in (1/2,1), \varepsilon \in (0, 1-\alpha)$ and 
\begin{equation}\label{ini-v}
v_0(0) = \psi_-, \quad v_0(1) = -\psi_+. 
\end{equation}
Suppose that $v_0$ and $\eta_0 \in \mathbb{R}$ satisfy 
\[ \| v_0\|_{C^\alpha([0,1])} + \eta_0 \le N \]
for some constant $N > 0$. 
Then there exist a positive constant $T_1$ depending only on $N$ such that a unique classical solution 
\[ v \in C^{\alpha, \alpha/2}([0,1] \times [0,T_1]) \cap C^\infty([0,1] \times (0,T_1]), \quad \eta \in C([0,T_1]) \cap C^\infty((0,T_1]) \]
to \eqref{normal-eq} with $v(\cdot,0) = v_0$ and $\eta(0) = \eta_0$ exists. 
Furthermore, the following estimates hold. 
\begin{itemize}
\item[(i)] There exists a positive constant $N_1$ depending only on $N$ such that 
\[ \|v(\cdot,\tau)\|_{C^\alpha([0,1])} + |\eta(\tau)| \le N_1, \quad \tau^{\frac{1-\alpha}{2}} \|v(\cdot,\tau)\|_{C^1([0,1])} \le N_1 \quad \mbox{for} \quad \tau \in [0,T_1]. \]
\item[(ii)] For any $T_0 \in (0,T_1)$ and $k \in \mathbb{N}$, there exists a positive constant $N_2$ depending only on $N, k$ and $T_0$ such that 
\[ \|v\|_{C^{k+\alpha, (k+\alpha)/2}([0,1] \times [T_0, T_1])} + \|\eta\|_{C^{(k + \alpha)/2}([T_0,T_1])} \le N_2. \]
\item[(iii)] If $(v,\eta), (\tilde{v}, \tilde{\eta})$ are solutions with initial data $(v_0, \eta_0), (\tilde{v}_0, \tilde{\eta}_0)$ satisfying \eqref{ini-v}, respectively, then there exists a positive constant $N_3$ depending only on $N$ such that for any $\tau \in (0, T_1]$, 
\[ \begin{aligned}
&\| v(\cdot, \tau) - \tilde{v}(\cdot,\tau)\|_{C^\alpha([0,1])} + \tau^{\frac{1-\alpha}{2}} \|v_y(\cdot,\tau) - \tilde{v}_y(\cdot,\tau)\|_{C([0,1])} + |\eta(\tau) - \tilde{\eta}(\tau)| \\
&\le N_3 (\|v_0 - \tilde{v}_0\|_{C^\alpha([0,1])} + |\eta_0 - \tilde{\eta}_0|). 
\end{aligned} \]
\end{itemize}
\end{lem}

\begin{proof}[Proof of Theorem \ref{thm:main1}]
We re-parametrize $X(\cdot,0)$ so that \eqref{restrict-para} with $t=0$ holds if necessary. 
From \eqref{as2}, the function $v(\cdot,0)$ constructed by the definition of the angle function \eqref{def-angle-fun} and the transformations \eqref{trans} satisfies $v(\cdot,0) \in C^1([0,1])$ and \eqref{ini-v}. 
Therefore, we can obtain a unique solution $(v,\eta)$ to \eqref{normal-eq} starting from $(v(\cdot,0), \log L(0))$ by applying Lemma \ref{short-v}. 
A classical solution $X$ to \eqref{apcf}--\eqref{bc2} satisfying \eqref{restrict-para} can be constructed from $(v,\eta)$ and $x(-1,0)$ by using the inverse of the transformations \eqref{trans} and parametrization \eqref{parametrization}, and we can see that the solution satisfies the regularity \eqref{reg-sol}. 
The uniqueness of $X$ follows from the uniqueness of $(v,\eta)$ obtained in Lemma \ref{short-v}. 
\end{proof}

\begin{rk}\label{rk:ene}
Using a similar argument for \eqref{deri-x}, we can obtain 
\begin{equation}\label{deri-x2}
\dfrac{\partial x}{\partial t} (1,t) = \dfrac{1}{\sin\psi_+} \left(\kappa(1,t) + \dfrac{\psi_+ + \psi_-}{L(t)}\right). 
\end{equation}
Substituting \eqref{deri-x} and \eqref{deri-x2} into \eqref{deri-ene}, we have
\[ \begin{aligned}
\dfrac{d}{dt} \left(L(t) - x(1,t) \cos \psi_+ + x(-1,t) \cos\psi_- \right) &\; = \dfrac{(\psi_+ + \psi_-)^2}{L(t)} - \int_{\gamma(t)} \kappa^2 \; ds \\
&\; \left( = - \int_{\gamma(t)} \left(\kappa + \dfrac{\psi_+ + \psi_-}{L(t)} \right)^2 \; ds \right). 
\end{aligned}\]
Therefore, the energy functional \eqref{def-ene} is a non-increasing function in time $t$. 
Furthermore, the area $A(t)$ defined by \eqref{def-area} is preserved with respect to time $t$ (see Lemma \ref{lem:area-pre}). 
As mentioned in the introduction, these structures indicate that a global-in-time solution to \eqref{apcf}--\eqref{bc2} converges to a minimizer of the following problem: 
\[ \min \{E(\gamma) : \gamma \; \mbox{satisfies the following properties (i) and (ii)} \}. \]
\begin{itemize}
\item[(i)] $\gamma$ is a simple curve having two endpoints on the $x$-axis. 
\item[(ii)] The area $A_\gamma$ defined by \eqref{def-area2} equals some constant $A > 0$. 
\end{itemize}
However, if we assume the existence of a minimizer (and also a critical point of $E$ under the area restriction), then the curvature of the curve should be constant and the curve will generate the interior contact angles $\psi_\pm$ on the endpoints. 
Since the contact angles at the intersections of a circle and a line should be the same, we obtain a contradiction if we assume $\psi_+ \neq \psi_-$. 
Therefore, there is no minimizer (and also no critical point) if $\psi_+ \neq \psi_-$. 

We can also see the nonexistence of minimizers (and also critical points) and the unboundedness of $E$ from a different perspective. 
Fix a curve $\gamma$ satisfying (i) and (ii). 
Then, the translated curves $\gamma + t \vec{e}_1$ also satisfy (i) and (ii) for any $t \in \mathbb{R}$, where $\vec{e}_1 = (1,0)$. 
It is easy to see that 
\[\begin{aligned} 
\dfrac{d}{dt} E(\gamma + t\vec{e}_1) =&\; \dfrac{d}{dt} (L(\gamma + t\vec{e}_1) - x_+(\gamma + t\vec{e}_1) \cos \psi_+ + x_-(\gamma + t\vec{e}_1) \cos \psi_-) \\
=&\; \cos \psi_- - \cos \psi_+. 
\end{aligned}\]
Therefore, the energy functional $E$ is changes linearly with respect to the translation of $\gamma$ if $\psi_+ \neq \psi_-$. 
This shows the nonexistence of minimizers (and also critical points) and the unboundedness of $E$ when $\psi_+ \neq \psi_-$. 
\end{rk}

\section{Preliminaries}\label{sec:pre} 

This section lists some necessary estimates and known results for proving the convergence of a global-in-time solution of \eqref{apcf}--\eqref{bc2} for a traveling wave. 

\subsection{Uniformly negativity of the curvature}

In this section, we study the preservation of area and concavity. 
First, we prove that the area $A(t)$ is preserved. 

\begin{lem}\label{lem:area-pre}
Assume \eqref{as1}--{\rm (A4)} and let $X$ be the solution to \eqref{apcf}--\eqref{bc2} obtained by Theorem \ref{thm:main1}. 
Then, 
\[
A(t) = A(0) \quad \mbox{for} \quad t \in (0,T). 
\]
\end{lem}

\begin{proof}
Let $V$ be the normal velocity $\kappa + (\psi_+ + \psi_-)/L(t)$ for simplicity. 
From \eqref{eq-Theta}, we obtain 
\begin{equation}\label{pre-area1} 
\dfrac{d}{dt} \mathcal{N} = \dfrac{d}{dt} (-\sin \Theta(p,t), \cos \Theta(p,t)) = -(\kappa_s + \alpha \kappa) \mathcal{T}. 
\end{equation}
Applying \eqref{bc1}, \eqref{bc2}, \eqref{tau-eq}, \eqref{bdry-ti-V}, \eqref{deri-leng}, \eqref{pre-area1}, the Frenet-Serret formulas $\mathcal{T}_s = \kappa \mathcal{N}$, $\mathcal{N}_s = - \kappa \mathcal{T}$ and integration by parts, we see that 
\[\begin{aligned} 
\dfrac{d}{dt} A(t) =&\; \dfrac{d}{dt} \dfrac{1}{2} \int_{\gamma(t)} \langle X, \mathcal{N} \rangle \; ds \\
=&\; \dfrac{1}{2} \int_{\gamma(t)} V - \langle X, (\kappa_s + \alpha \kappa) \mathcal{T} \rangle + \langle X, \mathcal{N} \rangle (- V \kappa + \alpha_s) \; ds \\
=&\; \dfrac{1}{2} \int_{\gamma(t)} V - \langle X, (V \mathcal{T})_s \rangle + \langle X, (\alpha \mathcal{N})_s \rangle \; ds \\
=&\; \int_{\gamma(t)} V \; ds - \dfrac{1}{2}\big\{(\alpha \sin \Theta + V \cos \Theta ) x(p,t)\big\} \Big|_{p=-1}^1 = 0. 
\end{aligned}\]
Therefore, we obtain the conclusion. 
\end{proof}

We will need the following isoperimetric inequality to prove the preservation of concavity. 

\begin{lem}
Assume \eqref{as1}--{\rm (A4)} and let $X$ be the solution to \eqref{apcf}--\eqref{bc2} obtained by Theorem \ref{thm:main1}. 
Then, 
\begin{equation}\label{iso-ine}
L(t) \ge \sqrt{2 \pi A(0)} \quad \mbox{for} \quad t \in [0,T). 
\end{equation}
\end{lem}

\begin{proof}
Let $\tilde{\gamma}(t)$ denote the mirror image of $\gamma(t)$ with respect to the $x$-axis, i.e., 
\[ \tilde{\gamma}(t) = \{\tilde{X}(p,t) = (x(p,t), -y(p,t)) \}, \]
where $(x(p,t), y(p,t))$ is the position of $\gamma(t)$. 
Then, the outward pointing unit normal vector $\tilde{\mathcal{N}}$ can be defined as 
\[ \tilde{\mathcal{N}} = (\mathcal{N}_x, -\mathcal{N}_y), \]
where $\mathcal{N}_x$ and $\mathcal{N}_y$ is the $x$-coordinate and the $y$-coordinate of the unit normal vector $\mathcal{N}$ of $\gamma(t)$. 
Furthermore, the (singed) area enclosed by $\gamma(t)$ and $\tilde{\gamma}(t)$ can be defined as 
\[ \dfrac{1}{2} \int_{\gamma(t)} \langle X, \mathcal{N} \rangle \; ds + \dfrac{1}{2} \int_{\tilde{\gamma(t)}} \langle \tilde{X}, \tilde{\mathcal{N}} \rangle \; ds \]
and it coincide with $2A(t)$ by virtue of the definition of $\tilde{\gamma}(t)$. 
Since $A(t)$ is preserved in time and the length of $\tilde{\gamma}(t)$ equals $2L(t)$, we obtain by the isoperimetric inequality for closed curves (see, for example, \cite{O}) 
\[ 2A(0) = 2A(t) \le \dfrac{(2L(t))^2}{4\pi}. \]
This implies the conclusion. 
\end{proof}

We now prove that the concavity is preserved for any solution. 

\begin{lem}\label{lem:concave}
Assume \eqref{as1}--{\rm (A4)} and let $X$ be the solution to \eqref{apcf}--\eqref{bc2} obtained by Theorem \ref{thm:main1}. 
$\Theta$ denotes the angle function defined by \eqref{def-angle-fun}. 
Then, 
\begin{align}
&-\psi_+ < \Theta(p,t) < \psi_- \quad \mbox{for} \quad (p,t) \in (-1,1) \times [0,T), \label{maxi-theta}\\
&\kappa(p, t) < 0 \quad \mbox{for} \quad (p,t) \in [-1,1] \times [0, T). \label{concavity}
\end{align}
\end{lem}

\begin{proof}
First, we prove \eqref{maxi-theta}. 
By virtue of \eqref{bc2} and {\rm (A4)}, the angle function $\Theta$ satisfies 
\begin{align}
& -\psi_+ \le \Theta(p,0) \le \psi_- \quad \mbox{for} \quad p \in [-1,1], \notag\\
& \Theta(\pm1,t) = \mp \psi_\pm \quad \mbox{for} \quad t \in [0,T). \label{max-bdr}
\end{align}
Furthermore, since $\Theta$ satisfies the differential equation \eqref{eq-Theta}, we can apply the maximum principle to obtain \eqref{maxi-theta} by virtue of \eqref{deri-theta}.

Next, we prove \eqref{concavity}. 
Since \eqref{maxi-theta} and \eqref{max-bdr} hold, we see that $p=1$ and $p=-1$ are the minimum point and the maximum point of $\Theta$, respectively. 
Therefore, applying the Hopf lemma, we obtain 
\begin{equation}\label{neg-kappa-bdr} 
\kappa(\pm1, t) < 0 \quad \mbox{for} \quad t \in (0,T) 
\end{equation}
by virtue of \eqref{deri-theta}. 
Now, let $W(p,t) := \kappa(p,t) e^{\mu t}$, where $\mu$ is a constant to be chosen latter. 
From the equation \eqref{k-s-eq}, we obtain the following differential equation for $W(p,t)$: 
\begin{equation}\label{W-eq}
W_t = W_{ss} + \alpha W_s + W\left(\kappa^2 + \dfrac{\psi_+ + \psi_-}{L(t)} \kappa + \mu \right). 
\end{equation}
If we choose $\mu > \dfrac{(\psi_+ + \psi_-)^2}{8 \pi A(0)}$, then we have by \eqref{iso-ine} 
\begin{equation}\label{W-nega}
\kappa^2 + \dfrac{\psi_+ + \psi_-}{L(t)} \kappa + \mu = \left(\kappa + \dfrac{\psi_+ + \psi_-}{2L(t)} \right)^2 + \mu - \dfrac{(\psi_+ + \psi_-)^2}{4(L(t))^2} \ge \mu - \dfrac{(\psi_+ + \psi_-)^2}{8 \pi A(0)} > 0. 
\end{equation}
Now fix a time $0 < t_1 <T$ and define 
\[ W_{\rm max}(t) := \sup_{p \in [-1,1]} W(p,t), \quad \beta(t_1) := \max\{ W_{\rm max}(0), \max_{t \in [0,t_1]} W(\pm1, t)\}. \]
Note that $\beta(t_1) < 0$ from {\rm (A4)} and \eqref{neg-kappa-bdr}. 
Suppose that at some time $t_2 \in (0,t_1]$, $\beta(t_1) < W_{\rm max}(t_2) < 0$ to obtain a contradiction. 
Let $t_{\rm max} \in (0,t_1]$ be the smallest time such that $W(p_{\rm max}, t_{\rm max}) = W_{\rm max} (t_2)$ for some interior point $p_{\rm max} \in (-1,1)$. 
Then, we have 
\[ W_t \ge 0, \quad W = W_{\rm max}(t_2) < 0, \quad W_s = 0, \quad W_{ss} \ge 0 \quad \mbox{at} \quad (p_{\rm max}, t_{\rm max}). \]
This contradicts with \eqref{W-eq} and \eqref{W-nega}, and we conclude that 
\[ \kappa(p,t) \le W_{\rm max}(t) e^{-\mu t} \le \beta(t_1) e^{-\mu t} < 0 \]
for any $p \in [-1, 1]$ and $0 < t < t_1 <T$.
\end{proof}

\begin{rk}\label{rk:pre-sim}
The argument for \eqref{concavity} is based on the argument in \cite{G}. 
The preservation of the convexity also holds for the flow of a Jordan curve governed by \eqref{apcf}. 
While the simplicity of a closed curve with single rotation index follows immediately from its convexity as in \cite{G}, we do not know whether the curve $\gamma(t)$ is simple or not in our problem under the assumptions \eqref{as1}--{\rm (A4)}. 
Indeed, by \eqref{concavity}, the simplicity of the concave curve $\gamma(t)$ is equivalent to the condition $x(-1,t) < x(1,t)$ and it seems that $x(-1,t) > x(1,t)$ could possibly hold for some initial curve when $\psi_+ + \psi_- > \pi$ (see Figure \ref{fig:non-simple}). 
At least in the case $\psi_+ + \psi_- \le \pi$, the condition $x(-1,t) < x(1,t)$ holds for any $t > 0$ by virtue of the comparison principle. 
If $\gamma(t)$ is not simple, then $A(t)$ does not coincide with the area enclosed by $\gamma(t)$ and the $x$-axis. 
\end{rk}

\begin{figure}[t]
\begin{center}
\scalebox{0.35}{\includegraphics{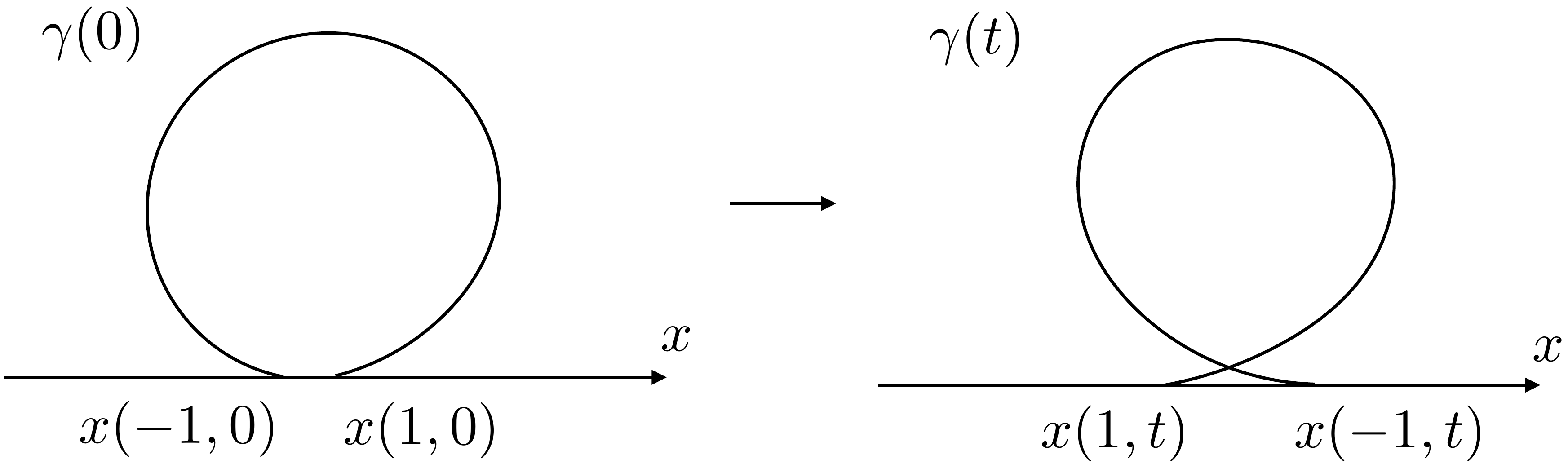}}
\end{center}
\caption{Possible case with a loss of the simplicity.}
\label{fig:non-simple}
\end{figure}

\begin{rk}\label{rk:regularity}
For the motion of a Jordan curve governed by \eqref{apcf}, the length of the curve is uniformly bounded by virtue of the variational structure as mentioned in the introduction. 
However, the variational structure for our problem is different from that for the motion of a Jordan curve. 
Therefore, we need some argument to prove the uniform boundedness of the length of the curve in our problem. 
For the motion of a concave graph, Shimojo and the author \cite{SK} proved the uniform boundedness of length. 
One of the key arguments in \cite{SK} for obtaining the boundedness is to compare the area $A(t)$ and the area of an interior triangle between $\gamma(t)$ and the $x$-axis. 
In that problem, the area $A(t)$ is larger than the area of an interior triangle, because $A(t)$ coincides with the area enclosed by $\gamma(t)$ and the $x$-axis. 
However, we cannot compare these areas in our problem because of the possibility of losing the simplicity of $\gamma(t)$ as in Remark \ref{rk:pre-sim}. 
The uniform boundedness of length was applied to obtain the uniform boundedness of the curvature in \cite{SK}. 
Therefore, we assume {\rm (A6)} and {\rm (A7)} in our problem, noting that the boundedness or otherwise of the length and the curvature should be studied further. 
\end{rk}

Since the curve $\gamma(t)$ is concave at any time $t$, the angle function $\Theta$ defined by \eqref{def-angle-fun} is injective. 
Therefore, we can re-parametrize $\gamma(t)$ by the angle $\theta \in [-\psi_+, \psi_-]$ between the tangent vector $\mathcal{T}$ and the $x$-axis. 
We will use this angle parameter $\theta$ from here on. 
As in \cite[Section 4]{SK}, we introduce the following differential equation for $\kappa$ in the variables $\theta$ and $t$: 
\begin{equation}\label{eq-theta}
\begin{cases}
\kappa_t = \kappa^2\left(\kappa_{\theta\theta} + \kappa + \frac{\psi_+ + \psi_-}{L(t)}\right) & \mbox{for} \quad -\psi_+ < \theta < \psi_-, \quad t>0, \\
\kappa_\theta = \cot \theta \left(\kappa + \frac{\psi_+ + \psi_-}{L(t)}\right) & \mbox{for} \quad \theta = \mp \psi_\pm, \quad t > 0, \\
\kappa(\theta,0) = \kappa_0(\theta) & \mbox{for} \quad -\psi_+ \le \theta \le \psi_- 
\end{cases}
\end{equation}
with the initial condition 
\begin{equation}\label{ini-con} 
\int_{-\psi_+}^{\psi_-} \frac{\sin\theta}{\kappa_0(\theta)} \; d\theta = 0. 
\end{equation}
Note that the length and the position of $\gamma(t)$ can be represented by the left endpoint $x(-1,0)$ at $t=0$ and the curvature $\kappa$ using the variables $\theta$ and $t$ as 
\begin{align}
& L(t) = - \int_{-\psi_+}^{\psi_-}\dfrac{d\theta}{\kappa(\theta,t)}, \label{leng-theta} \\
& x(-1,t) = x(-1,0) - \dfrac{1}{\sin \psi_-} \int_0^t \kappa(\psi_-, \tilde{t}) + \dfrac{\psi_+ + \psi_-}{L(\tilde{t})} \; d\tilde{t}, \label{left-theta}\\
& \gamma(t) = \left\{ \left( x(-1,t) - \int_{\theta}^{\psi_-} \dfrac{\cos\tilde{\theta}}{\kappa (\tilde{\theta}, t)} \; d\tilde{\theta}, -\int_{\theta}^{\psi_-} \dfrac{\sin\tilde{\theta}}{\kappa(\tilde{\theta}, t)} \; d\tilde{\theta} \right) : -\psi_+ \le \theta \le \psi_- \right\}. \label{posi-theta}
\end{align}
From the representation \eqref{posi-theta}, we see that the condition \eqref{ini-con} is equivalent to the boundary condition \eqref{bc1} at $t=0$. 
Furthermore, it is easy to see that 
\begin{equation}\label{bc1-theta} 
\int_{-\psi_+}^{\psi_-} \dfrac{\sin \theta}{\kappa (\theta,t)} \; d\theta = 0 \quad \mbox{for} \quad t \ge 0 
\end{equation}
since the derivative of the integral is zero for any solution to \eqref{eq-theta}. 
This condition is equivalent to \eqref{bc1}. 
Moreover, according to \eqref{posi-theta}, \eqref{bc1-theta} and $\mathcal{N} = (-\sin \theta, \cos\theta)$, the area $A(t)$ defined by \eqref{def-area} can be represented using $\kappa(\theta,t)$ as 
\begin{equation}\label{area-theta}
A(t) = \dfrac{1}{2} \left(-\int_{-\psi_+}^{\psi_-} \dfrac{\sin\theta}{\kappa(\theta, t)} \int_{\theta}^{\psi_-} \dfrac{\cos\tilde{\theta}}{\kappa(\tilde{\theta}, t)} \; d\tilde{\theta}d\theta  + \int_{-\psi_+}^{\psi_-} \dfrac{\cos \theta}{\kappa (\theta, t)} \int_{\theta}^{\psi_-} \dfrac{\sin \tilde{\theta}}{\kappa(\tilde{\theta}, t)} \; d\tilde{\theta} d\theta \right).
\end{equation}

In order to prove the global stability of the traveling waves obtained by \cite{KK}, we find a Lyapunov functional for the differential equation \eqref{eq-theta}. 
In this argument, we will need the boundedness of $|\kappa_{\theta}|$ and the uniform negativity of $\kappa$. 
We remark that Lemma \ref{lem:concave} only shows the negativity of $\kappa$ at any positive time. 
Hence we cannot obtain the uniform boundedness of $|\kappa_\theta|$ even if $|\kappa_s|$ is uniformly bounded since $\frac{\partial}{\partial \theta} = \frac{1}{\kappa} \cdot \frac{\partial}{\partial s}$ holds. 
The proof of the following lemma is based on \cite{A, CLW}.

\begin{lem}\label{lem:lip-theta}
Assume \eqref{as1}--{\rm (A7)} and let $X$ be the solution to \eqref{apcf}--\eqref{bc2} obtained by Theorem \ref{thm:main1}. 
Then, for any $\varepsilon > 0$, there exists a constant $M_1 > 0$ such that 
\begin{equation}\label{lip-theta}
|\kappa_\theta ( \theta, t)| \le M_1 \quad \mbox{for} \quad (\theta,t) \in [-\psi_+, \psi_-] \times [\varepsilon,\infty). 
\end{equation}
\end{lem}

\begin{proof}
Denote $I:= [-\psi_+, \psi_-]$ for simplicity. 
We note that $\|\kappa_\theta(\cdot,\varepsilon)\|_{L^\infty(I)}$ is bounded by some constant depending on $\varepsilon$ by virtue of the higher regularity in Lemma \ref{short-v} and the assumptions {\rm (A6)} and {\rm (A7)}. 
Here, \eqref{concavity} is also used since the relation $\frac{\partial}{\partial \theta} = \frac{1}{\kappa} \cdot \frac{\partial}{\partial s}$ holds. 
We prove 
\begin{equation}\label{lip-theta-1}
\sup_{I\times [\varepsilon,\infty)}(\kappa^2 + \kappa_\theta^2) \le \max \left\{ \sup_{I \times [\varepsilon,\infty)} \kappa^2, \sup_{I\times \{\varepsilon\}} (\kappa^2 + \kappa_\theta^2), \sup_{\{\mp \psi_\pm\} \times [\varepsilon,\infty)} (\kappa^2 + \kappa_\theta^2) \right\}. 
\end{equation}
By virtue of the assumption {\rm (A6)}, the isoperimetric inequality \eqref{iso-ine} and the boundary condition of \eqref{eq-theta}, we may see that the right hand side of \eqref{lip-theta-1} is bounded as 
\begin{align}
&\sup_{I \times [\varepsilon,\infty)} \kappa^2 \le c_1^2, \label{bdd-const1}\\
&\sup_{I \times \{\varepsilon\}} (\kappa^2 + \kappa_\theta^2) \le c_1^2 + \|\kappa_\theta (\cdot,\varepsilon)\|_{L^\infty(I)}^2, \label{bdd-const2}\\
&\sup_{\{\mp \psi_\pm\} \times [\varepsilon,\infty)} (\kappa^2 + \kappa_{\theta}^2) \le c_1^2 + \max\{ \cot^2 \psi_+, \cot^2 \psi_-\} \left(c_1 + \frac{\psi_+ + \psi_-}{\sqrt{2 \pi A(0)} }\right)^2. \label{bdd-const3}
\end{align} 
Therefore, it is enough to prove \eqref{lip-theta-1} in order to obtain the conclusion.
Let $\Psi(\theta,t) = \kappa^2(\theta,t) + \kappa_\theta^2(\theta,t)$ and let $T > \varepsilon$ be fixed. 
Suppose at $(\theta_0, t_0) \in I \times [\varepsilon,T]$ we have $\Phi(\theta_0, t_0) = \sup_{I \times [\varepsilon,T]} (\kappa^2 + \kappa_\theta^2)$. 
Then we may assume $\theta_0$ is an interior point in $I$ and $t_0 > \varepsilon$ since otherwise we are done. 
We claim that $\kappa_\theta (\theta_0, t_0) = 0$. 
If not, we will have 
\begin{align}
&\frac{\Phi_\theta}{2\kappa_\theta} = \kappa_{\theta\theta} + \kappa = 0, \label{lip-theta-2}\\
&\Phi_{\theta\theta} = 2\kappa_{\theta\theta} ( \kappa_{\theta\theta} + \kappa) + 2\kappa_{\theta} (\kappa_{\theta\theta\theta} + \kappa_\theta) = 2\kappa_{\theta} (\kappa_{\theta\theta\theta} + \kappa_\theta) \le 0, \label{lip-theta-3}\\
& \Phi_t \ge 0 \label{lip-theta-4}
\end{align}
at $(\theta_0,t_0)$. 
On the other hand, by virtue of \eqref{eq-theta}, we obtain by a simple calculation 
\begin{equation}\label{lip-theta-5}
\Phi_t = 2 \kappa^2 \kappa_\theta (\kappa_{\theta\theta\theta} + \kappa_\theta) + (2 \kappa^3 + 4\kappa_\theta^2 \kappa) \dfrac{\psi_+ + \psi_-}{L(t)} + (4\kappa \kappa_\theta^2 + 2\kappa^3) (\kappa_{\theta\theta} + \kappa) 
\end{equation}
for $(\theta, t) \in I \times [0,T]$. 
Since $\kappa$ is negative at finite time $t_0$, substituting \eqref{lip-theta-2} and \eqref{lip-theta-3} into \eqref{lip-theta-5}, we obtain $\Phi_t < 0$. 
This contracts with \eqref{lip-theta-4}. 
Therefore $\kappa_\theta (\theta_0, t_0) = 0$ and we conclude that 
\[ \sup_{I \times [0,T]} (\kappa^2 + \kappa_\theta^2) = \kappa^2(\theta_0, t_0) \le \sup_{I \times [0,T]} \kappa^2. \]
Thus, we obtain \eqref{lip-theta-1}. 
The estimate \eqref{lip-theta} follows form \eqref{lip-theta-1}--\eqref{bdd-const3} and $\sup_{I \times [\varepsilon, \infty)} \kappa_\theta^2 \le \sup_{I \times [\varepsilon, \infty)} \kappa^2 + \kappa_\theta^2$. 
\end{proof}

Next, we prove the uniform negativity of the curvature. 

\begin{lem}\label{lem:uni-nega-k}
Assume \eqref{as1}--{\rm (A7)} and let $X$ be the solution to \eqref{apcf}--\eqref{bc2} obtained by Theorem \ref{thm:main1}. 
Then, there exists a constant $M_2 > 0$ such that
\begin{equation}\label{uni-nega-k} 
\kappa(\theta,t) \le - M_2 \quad \mbox{for} \quad (\theta,t) \in [-\psi_+, \psi_-] \times [0,\infty). 
\end{equation}
\end{lem}

\begin{proof}
In proof by contradiction, suppose that there exists a sequence $(\theta_i, t_i) \in [-\psi_+,\psi_-] \times [0,\infty)$ such that 
\begin{equation}\label{uni-nega-k-1}
\lim_{i \to \infty} \kappa(\theta_i, t_i) = 0. 
\end{equation}
Since $[-\psi_+, \psi_-]$ is compact and $\kappa$ is negative at finite time, we may assume 
\[ \lim_{i \to \infty} \theta_i = \hat{\theta}, \quad \lim_{i \to \infty} t_i = \infty \]
for some $\hat{\theta} \in [-\psi_+, \psi_-]$ without loss of generality. 
First we consider the case $\hat{\theta} \neq -\psi_+$. 
Let $s(\theta, t)$ is the arc-length at $(x(\theta,t), y(\theta,t)) \in \gamma(t)$. 
By virtue of $\frac{\partial}{\partial \theta} = \frac{1}{\kappa} \cdot \frac{\partial}{\partial s}$ and {\rm (A7)}, we may see that 
\begin{equation}\label{uni-nega-k-2}
-\int_{-\psi_+}^{\theta_i} \dfrac{d\theta}{\kappa(\theta,t_i)} = |s(\theta_i, t_i) - s(-\psi_+,t_i)| \le L(t_i) \le c_2
\end{equation}
for arbitrary $i \in \mathbb{N}$. 
Now, we prove the left hand side of \eqref{uni-nega-k-2} diverges to infinity as $i \to \infty$. 
Applying \eqref{lip-theta}, we have 
\[ -\kappa(\theta,t_i) = (\kappa(\theta_i,t_i) - \kappa(\theta, t_i)) - \kappa(\theta_i,t_i) \le M_1 |\theta - \theta_i| - \kappa(\theta_i, t_i) = M_1 (\theta_i - \theta) - \kappa(\theta_i,t_i) \]
for arbitrary $\theta \in [-\psi_+, \theta_i]$ and sufficiently large $i \in \mathbb{N}$. 
Therefore, we obtain 
\begin{equation}\label{uni-nega-k-3}
-\int_{-\psi_+}^{\theta_i} \dfrac{d\theta}{\kappa(\theta,t_i)} \ge \int_{-\psi_+}^{\theta_i} \dfrac{d\theta}{C_1(\theta_i-\theta) - \kappa(\theta_i,t_i)} = \dfrac{1}{M_1} \log \left(\dfrac{M_1(\theta_i + \psi_+)- \kappa(\theta_i,t_i)}{-\kappa(\theta_i,t_i)} \right)
\end{equation}
for sufficiently large $i \in \mathbb{N}$. 
By virtue of the convergence \eqref{uni-nega-k-1} and $\lim_{i \to \infty} \theta_i = \hat{\theta} \neq -\psi_+$, the right hand side of \eqref{uni-nega-k-3} diverges to infinity, and hence the left hand side also diverges to infinity. 
This contradicts with \eqref{uni-nega-k-2}. 
In the case $\hat{\theta} \neq \psi_-$, we can obtain a contradiction by a similar argument. 
Therefore, we obtain \eqref{uni-nega-k}. 
\end{proof}

\subsection{Traveling waves}

As mentioned in the introduction, the existence of traveling waves was studied by \cite{KK, SK}. 
We recall the existence theory from \cite{KK} for convenience. 

\begin{thm}[{\cite[Theorem 1]{KK}}] \label{thm:ex-traveling}
For any $\psi_\pm \in (0,\pi)$, there exists a traveling wave $\mathcal{W}(t)$ for \eqref{apcf}--\eqref{bc2} defined as \eqref{def-tw} such that 
\[ -\int_{\mathcal{W}(0)} \kappa_{\mathcal{W}} \; ds = \psi_+ + \psi_-, \quad L_{\mathcal{W}} = 1, \]
where $\kappa_{\mathcal{W}}$ and $L_{\mathcal{W}}$ are the curvature and the length of $\mathcal{W}(0)$, respectively. 
Furthermore, the traveling wave is concave {\rm (}i.e.\ $\kappa_{\mathcal{W}} < 0${\rm )} and unique except the translation parallel to the $x$-axis. 
In addition, the wave speed $c$ and the curvature $\kappa_{\mathcal{W}}$ fulfill
\begin{align}
&\psi_{-}
\left\{
\begin{array}{l}
> \\ = \\ <
\end{array}
\right\}
\psi_{+}
\;
\iff
\;
c
\left\{
\begin{array}{l}
> \\ = \\ <
\end{array}
\right\}
0, \label{sign-c}\\
&\kappa_{\mathcal{W}} (\theta) = -c \sin \theta - (\psi_+ + \psi_-) \;\; \mbox{for} \;\; -\psi_+ \le \theta \le \psi_-. \label{curvature-tra}
\end{align}
\end{thm}

We next study the positivity of the area $A_{\mathcal{W}}$ defined by \eqref{def-area3} for the traveling wave obtained in Theorem \ref{thm:ex-traveling}. 
The positivity is not obvious, because we do not know whether the profile curve $\mathcal{W}(0)$ is simple (see also Remark \ref{rk:pre-sim}). 

\begin{lem}
Let $\mathcal{W}(t)$ be the traveling wave obtained in Theorem \ref{thm:ex-traveling}. 
Then, 
\begin{equation}\label{posi-area-tra} 
x_{-, \mathcal{W}} < x_{+, \mathcal{W}}, \quad A_{\mathcal{W}} > 0, 
\end{equation}
where $x_{-, \mathcal{W}}$ is the left endpoint of $\mathcal{W}(0)$, $x_{+, \mathcal{W}}$ is the right endpoint of $\mathcal{W}(0)$ and $A_{\mathcal{W}}$ is the area defined by \eqref{def-area3}. 
\end{lem}

\begin{proof}
When $\psi_+ = \psi_-$, \eqref{posi-area-tra} obviously holds since $\mathcal{W}(0)$ is an arc. 
Thus, we consider only in the case $\psi_+ \neq \psi_-$. 
According to the representation as \eqref{posi-theta}, we can see 
\[ x_{+, \mathcal{W}} = x_{-,\mathcal{W}} - \int_{-\psi_+}^{\psi_-} \dfrac{\cos\theta}{\kappa_{\mathcal{W}}(\theta)} \; d\theta. \]
Therefore, we obtain by virtue of \eqref{curvature-tra} 
\begin{equation}\label{posi-area-tra1} 
x_{+, \mathcal{W}} - x_{-,\mathcal{W}} = \int_{-\psi_+}^{\psi_-} \dfrac{\cos\theta}{c \sin\theta + (\psi_+ + \psi_-)} \; d\theta = \dfrac{1}{c} \log \left(\dfrac{c \sin \psi_- + (\psi_+ + \psi_-)}{-c \sin \psi_+ + (\psi_+ + \psi_-)} \right). 
\end{equation}
Here, the negativity of $\kappa_{\mathcal{W}} = - c \sin \theta - (\psi_+ + \psi_-)$ have used. 
From the sign of $c$ as in \eqref{sign-c}, we obtain 
\[\begin{aligned} 
&\dfrac{c \sin \psi_- + (\psi_+ + \psi_-)}{-c \sin \psi_+ + (\psi_+ + \psi_-)} > 1 \quad \mbox{if} \quad \psi_+ < \psi_-, \\
&\dfrac{c \sin \psi_- + (\psi_+ + \psi_-)}{-c \sin \psi_+ + (\psi_+ + \psi_-)} < 1 \quad \mbox{if} \quad \psi_+ > \psi_-.
\end{aligned}\]
Substituting the above inequality into \eqref{posi-area-tra1} and applying \eqref{sign-c} again, we conclude $x_{-, \mathcal{W}} < x_{+, \mathcal{W}}$. 
Hence, the profile curve $\mathcal{W}(0)$ is simple and $A_{\mathcal{W}}$ is positive since it coincides with the area enclosed by $\mathcal{W}(0)$ and the $x$-axis. 
\end{proof}

When a pair $(\mathcal{W}(0), c)$ constructs a traveling wave as \eqref{def-tw}, scaled and translated traveling waves 
\[ \mathcal{W}_{\lambda, a} (t) := \lambda \mathcal{W}(0) + (c/\lambda)t \vec{e}_1 + a \vec{e}_1 \]
can be constructed for any $\lambda > 0$ and any $a \in \mathbb{R}$, where $\vec{e}_1 = (1,0)$. 
By virtue of the positivity of $A_{\mathcal{W}}$ and the invariance under scaling and translation, we can rewrite Theorem \ref{thm:ex-traveling} as the following corollary so that it can be applied to our problem. 

\begin{cor}\label{cor:ex-tw}
For any $A^* > 0$, there exists a traveling wave $\mathcal{W}(t)$ for \eqref{apcf}--\eqref{bc2} defined as \eqref{def-tw} such that $A_{\mathcal{W}} = A^*$ and 
\[ -\int_{\mathcal{W}(0)} \kappa_{\mathcal{W}} \; ds = \psi_+ + \psi_-, \]
where $\kappa_{\mathcal{W}}$ is the curvature of $\mathcal{W}(0)$ and $A_{\mathcal{W}}$ is defined by \eqref{def-area3}. 
Furthermore, the traveling wave is concave and unique except the translation parallel to the $x$-axis. 
In addition, the wave speed $c$ and the curvature $\kappa_{\mathcal{W}}$ fulfill \eqref{sign-c} and 
\begin{equation}\label{curvature-tw} 
\kappa_{\mathcal{W}}(\theta) = - c \sin \theta - \dfrac{\psi_+ + \psi_-}{L_{\mathcal{W}}} \quad \mbox{for} \quad - \psi_+ \le \theta \le \psi_-, 
\end{equation}
where $L_{\mathcal{W}}$ is the length of $\mathcal{W}(0)$. 
\end{cor}

\begin{rk}\label{rk:posi-area-TW}
According to Theorem \ref{thm:ex-traveling} and \eqref{posi-area-tra}, there is no traveling wave such that 
\[ -\int_{\mathcal{W}(0)} \kappa_{\mathcal{W}} \; ds = \psi_+ + \psi_-, \quad A_{\mathcal{W}} \le 0. \]
However, the area $A(0)$ can be non-positive if we do not assume the concavity or simplicity of the initial curve $\gamma(0)$ (see Figure \ref{fig:non-positive}). 
In this case, the asymptotic behavior of a global-in-time solution and its singularities are not obvious. 
\end{rk}

\begin{figure}[t]
\begin{center}
\scalebox{0.35}{\includegraphics{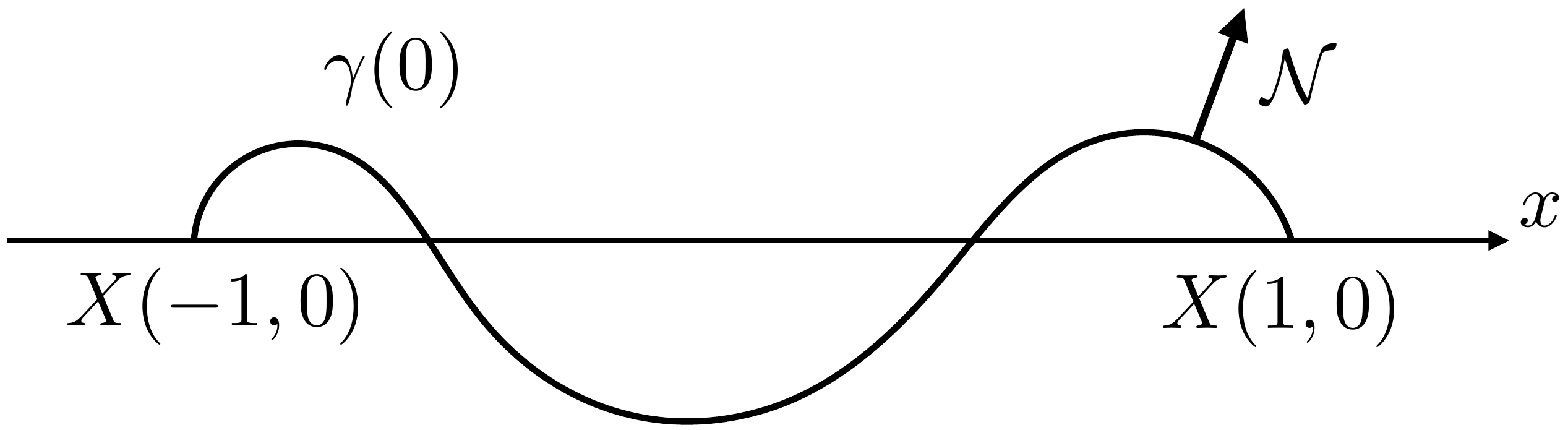}}
\end{center}
\caption{An example of a non-concave graph with negative area.}
\label{fig:non-positive}
\end{figure}

In the present paper, our aim is to prove the convergence of a global-in-time solution for the traveling wave without the assumption of ``closeness'' between the initial curve and a traveling wave. 
The local exponential stability of the traveling wave was proved by \cite{SK} when $\psi_\pm \in (0,\pi/2)$, which is required a ``closeness'' between the solution and a traveling wave. 
The proof of the local exponential stability in \cite{SK} is based on the spectral analysis of a linearized problem for a generalized differential equation of \eqref{eq-theta} around the stationary solution $\kappa_{\mathcal{W}}$ obtained in \eqref{curvature-tw}. 
One of the key properties for proving that the spectrum of the linearized problem consists of non-positive eigenvalues is that the zero level set of $\sin \theta$ is a simple point in $[-\psi_+, \psi_-]$. 
This property of the sine function holds for general angle conditions $\psi_\pm \in (0,\pi)$. 
Therefore, the local exponential stability in \cite{SK} can be expanded to the more general case $\psi_\pm \in (0, \pi)$ as in the following theorem. 
The proof of the following theorem is the same as the proof in \cite{SK}. 

\begin{thm}[{\cite[Theorem 1.2, Theorem 5.1]{SK}}]\label{thm:local-sta} 
Let $\gamma(t)$ and $\mathcal{W}(t)$ be a global-in-time solution to \eqref{apcf}--\eqref{bc2} and the traveling wave obtained by \cite{KK} such that $A(0) = A_{\mathcal{W}}$, respectively. 
Denote $\kappa(\theta,t)$ and $\kappa_{\mathcal{W}}(\theta)$ by the curvature of $\gamma(t)$ and $\mathcal{W}(0)$, respectively. 
Assume $\|\kappa(\cdot,0) - \kappa_{\mathcal{W}}\|_{L^\infty}$ is sufficiently small. 
Then, there exists a constant $a \in \mathbb{R}$ such that, as $t \to \infty$, $\kappa(\cdot,t)$ and $\gamma(t)$ exponentially converge to $\kappa_{\mathcal{W}}$ in $C^\infty$ and $\mathcal{W}(t) + a \vec{e}_1$ in the Hausdorff metric, respectively, where $\vec{e}_1 = (1,0)$. 
\end{thm}

The convergence of $\gamma(t)$ can be easily seen from the convergence of $\kappa$ in the representations \eqref{left-theta} and \eqref{posi-theta}. 
According to Theorem \ref{thm:local-sta}, it is enough to prove that the curvature of a global-in-time solution is sufficiently close to the curvature of a traveling wave at some time to prove the global stability of the traveling waves. 

\section{Global stability of traveling waves} \label{sec:main}
In this section, we prove the global stability of the traveling waves obtained by Corollary \ref{cor:ex-tw} in the sense of Theorem \ref{thm:main2}. 
Under the assumptions {\rm (A4)}--{\rm (A7)}, $\kappa(\cdot,t)$ is uniformly bounded in $C^\infty([-\psi_+, \psi_-])$ by virtue of \eqref{uni-nega-k} and $\frac{\partial}{\partial \theta} = \frac{1}{\kappa} \cdot \frac{\partial}{\partial s}$. 
Therefore, we can apply the Ascoli-Arzel\'a theorem to $\kappa(\theta,t)$. 
In the following, we study the $\omega$-limit points of $\{\kappa(\theta,t)\}_{t \ge 0}$. 
 
We note that the variable $\theta$ and the curvature $\kappa$ does not depend on the translation of $\gamma(t)$ parallel to the $x$-axis. 
As mentioned in Remark \ref{rk:ene}, the functional $E(\gamma)$ is not bounded from blow because of the dependence of $E(\gamma)$ on the translation. 
Therefore, we find a Lyapunov functional of $\kappa(\theta,t)$ to analyze the $\omega$-limit points of $\{\kappa(\theta,t)\}_{t \ge 0}$. 

\begin{defi}
Define functionals $F_1$ and $F_2$ as 
\[\begin{aligned} 
F_1(t) :=&\; \int_{-\psi_+}^{\psi_-} -\dfrac{1}{2} \kappa_{\theta}^2(\theta,t) + \dfrac{1}{2} \kappa^2(\theta,t) + \kappa(\theta,t)\dfrac{\psi_+ + \psi_-}{L(t)} \; d\theta \\
&\; +\cot \psi_- \left(\dfrac{1}{2} \kappa^2(\psi_-,t) + \kappa(\psi_-,t) \dfrac{\psi_+ + \psi_-}{L(t)} \right) \\
&\; + \cot\psi_+ \left(\dfrac{1}{2}\kappa^2(-\psi_+,t) + \kappa(-\psi_+,t) \dfrac{\psi_+ + \psi_-}{L(t)}\right), \\
F_2(t) :=&\; \dfrac{(\psi_+ + \psi_-)^2}{2(L(t))^2} \{ (\psi_+ + \psi_-) + \cot \psi_- + \cot \psi_+ \}. 
\end{aligned}\]
Let $F(t)$ be the sum of $F_1(t)$ and $F_2(t)$ and denote $\tilde{F}(t)$ by 
\begin{equation}\label{def-lia} 
\tilde{F}(t) := (L(t))^2{\rm exp}\left(-\dfrac{2\int_{-\psi_+}^{\psi_-} \log(-\kappa) \; d\theta}{\psi_++\psi_-} \right) F(t). 
\end{equation}
\end{defi}

We prove $-\tilde{F}(t)$ is a Lyapunov function. 
Here, we note that H\"{o}lder's inequality implies 
\begin{equation}\label{holder-ine}
\left(\int_{-\psi_+}^{\psi_-} \dfrac{\kappa_t}{\kappa} \; d\theta \right)^2 \le (\psi_+ + \psi_-)\int_{-\psi_+}^{\psi_-} \dfrac{\kappa_t^2}{\kappa^2} \; d\theta. 
\end{equation}

\begin{lem}\label{lem:liapunov}
For $t \in (0,\infty)$, the functional $\tilde{F}(t)$ satisfies 
\begin{equation}\label{liapunov} 
\dfrac{d}{dt} \tilde{F}(t) = (L(t))^2{\rm exp}\left(-\dfrac{2\int_{-\psi_+}^{\psi_-} \log(-\kappa) \; d\theta}{\psi_++\psi_-} \right) \left(\int_{-\psi_+}^{\psi_-} \dfrac{\kappa_t^2}{\kappa^2} \; d\theta - \dfrac{1}{\psi_++\psi_-} \left(\int_{-\psi_+}^{\psi_-} \dfrac{\kappa_t}{\kappa} \; d\theta \right)^2\right). 
\end{equation}
\end{lem}

\begin{proof}
We obtain by applying \eqref{eq-theta} and integration by parts 
\begin{equation}\label{deri-lia}
\begin{aligned}
\dfrac{d}{dt} F_1(t) =&\; \int_{-\psi_+}^{\psi_-} \kappa_t \left(\kappa_{\theta\theta} + \kappa + \dfrac{\psi_+ + \psi_-}{L(t)}\right) \; d\theta \\
&\; - \dfrac{\psi_+ + \psi_-}{(L(t))^2} L'(t) \left( \int_{-\psi_+}^{\psi_-} \kappa \; d\theta + \kappa(\psi_-,t) \cot\psi_- + \kappa(-\psi_+, t) \cos\psi_+ \right) \\
=&\; \int_{-\psi_+}^{\psi_-} \dfrac{\kappa_t^2}{\kappa^2}\; d\theta - \dfrac{\psi_+ + \psi_-}{(L(t))^2} L'(t) \left( \int_{-\psi_+}^{\psi_-} \kappa \; d\theta + \kappa(\psi_-,t) \cot\psi_- + \kappa(-\psi_+, t) \cos\psi_+ \right). 
\end{aligned}
\end{equation}
We calculate the second term of the right hand side of \eqref{deri-lia}. 
By applying \eqref{eq-theta} and integration by parts, we have 
\begin{equation}\label{lia-1}
\int_{-\psi_+}^{\psi_-} \dfrac{\kappa_t}{\kappa} \; d\theta = 2 F_1(t) - \dfrac{\psi_+ + \psi_-}{L(t)} \left( \int_{-\psi_+}^{\psi_-} \kappa \; d\theta + \kappa (\psi_-, t) \cot\psi_- + \kappa(-\psi_+, t) \cot\psi_+ \right).
\end{equation}
Furthermore, \eqref{lia-1} and \eqref{eq-theta} imply
\begin{equation}\label{lia-2}
\begin{aligned}
\dfrac{L'(t)}{L(t)} =&\; \dfrac{1}{L(t)} \int_{-\psi_+}^{\psi_-} \dfrac{\kappa_t}{\kappa^2} \; d\theta \\
=&\; \dfrac{1}{L(t)} \left(\int_{-\psi_+}^{\psi_-} \kappa \; d\theta + \kappa(\psi_-,t) \cot \psi_- + \kappa(-\psi_+,t) \cot\psi_+ \right) + \dfrac{2}{\psi_++\psi_-} F_2(t) \\
=&\; \dfrac{2F(t)}{\psi_+ + \psi_-} - \dfrac{1}{\psi_+ + \psi_-} \int_{-\psi_+}^{\psi_-} \dfrac{\kappa_t}{\kappa} \; d\theta. 
\end{aligned}
\end{equation}
By substituting \eqref{lia-1} and \eqref{lia-2} into \eqref{deri-lia}, we obtain 
\[\begin{aligned} 
\dfrac{d}{dt} F_1(t) =&\; \int_{-\psi_+}^{\psi_-} \dfrac{\kappa_t^2}{\kappa^2} \; d\theta - 2\dfrac{L'(t)}{L(t)} F_1(t) + \dfrac{L'(t)}{L(t)} \int_{-\psi_+}^{\psi_-} \dfrac{\kappa_t}{\kappa} \; d\theta \\
=&\; \int_{-\psi_+}^{\psi_-} \dfrac{\kappa_t^2}{\kappa^2} \; d\theta - 2\dfrac{L'(t)}{L(t)} F_1(t) + \dfrac{2F(t)}{\psi_+ + \psi_-}\int_{-\psi_+}^{\psi_-} \dfrac{\kappa_t}{\kappa} \; d\theta - \dfrac{1}{\psi_++\psi_-} \left( \int_{-\psi_+}^{\psi_-} \dfrac{\kappa_t}{\kappa} \; d\theta \right)^2. 
\end{aligned}\]
Since $F_2'(t)$ coincides with $-2L'(t)F_2(t)/ L(t)$, the equality is equivalent to 
\[
\dfrac{d}{dt} F(t) + 2\dfrac{L'(t)}{L(t)} F(t) - \dfrac{2F(t)}{\psi_+ + \psi_-} \int_{-\psi_+}^{\psi_-} \dfrac{\kappa_t}{\kappa} \; d\theta = \int_{-\psi_+}^{\psi_-} \dfrac{\kappa_t^2}{\kappa^2} - \dfrac{1}{\psi_++\psi_-} \left(\int_{-\psi_+}^{\psi_-} \dfrac{\kappa_t}{\kappa} \; d\theta \right)^2. 
\]
This concludes \eqref{liapunov}. 
\end{proof}

Since the equality holds in \eqref{holder-ine} only if $\kappa_t = \alpha \kappa$ for some $\alpha \in \mathbb{R}$, we need the following lemma to analyze the $\omega$-limit points of $\{\kappa(\theta,t)\}_{t \ge 0}$. 

\begin{lem}\label{lem:lim-kappa}
Assume a negative valued function $\hat{\kappa}\in C^\infty([-\psi,\psi_+])$ and a constant $\alpha \in \mathbb{R}$ satisfies 
\begin{align}
& \alpha \hat{\kappa} = \hat{\kappa}^2 \left( \hat{\kappa}_{\theta \theta} + \hat{\kappa} + \frac{\psi_+ + \psi_-}{L(\hat{\kappa})}\right) \quad \mbox{for} \quad -\psi_+ < \theta < \psi_-, \label{as-lim-kappa-1}\\
& \hat{\kappa}_{\theta} = \cot \theta \left(\hat{\kappa} + \frac{\psi_+ + \psi_-}{L(\hat{\kappa})}\right) \quad \mbox{for} \quad \theta = \mp \psi_\pm, \label{as-lim-kappa-2} \\
& \int_{-\psi_+}^{\psi_-} \dfrac{\sin\theta}{\hat{\kappa}(\theta)} \; d\theta = 0, \label{as-lim-kappa-3} \\
& \dfrac{1}{2} \left(-\int_{-\psi_+}^{\psi_-} \dfrac{\sin\theta}{\hat{\kappa}(\theta)} \int_{\theta}^{\psi_-} \dfrac{\cos\tilde{\theta}}{\hat{\kappa}(\tilde{\theta})} \; d\tilde{\theta}d\theta  + \int_{-\psi_+}^{\psi_-} \dfrac{\cos \theta}{\hat{\kappa} (\theta)} \int_{\theta}^{\psi_-} \dfrac{\sin \tilde{\theta}}{\hat{\kappa}(\tilde{\theta})} \; d\tilde{\theta} d\theta \right) > 0, \label{as-lim-kappa-4}
\end{align}
where $L(\hat{\kappa}) = -\int_{-\psi_+}^{\psi_-} \frac{d\theta}{\hat{\kappa}}$. 
Then, $\alpha$ is $0$ and $\hat{\kappa}$ coincides with $\kappa_{\mathcal{W}}$ obtained in Corollary \ref{cor:ex-tw} with 
\begin{equation}\label{area-lim-kappa} 
A^* = \dfrac{1}{2} \left(-\int_{-\psi_+}^{\psi_-} \dfrac{\sin\theta}{\hat{\kappa}(\theta)} \int_{\theta}^{\psi_-} \dfrac{\cos\tilde{\theta}}{\hat{\kappa}(\tilde{\theta})} \; d\tilde{\theta}d\theta  + \int_{-\psi_+}^{\psi_-} \dfrac{\cos \theta}{\hat{\kappa} (\theta)} \int_{\theta}^{\psi_-} \dfrac{\sin \tilde{\theta}}{\hat{\kappa}(\tilde{\theta})} \; d\tilde{\theta} d\theta \right). 
\end{equation}
\end{lem}

\begin{proof}
Let $\hat{\gamma}$ be the plane curve constructed as
\[ \hat{\gamma} = \left\{ \left( - \int_{\theta}^{\psi_-} \dfrac{\cos\tilde{\theta}}{\hat{\kappa} (\tilde{\theta}, t)} \; d\tilde{\theta}, -\int_{\theta}^{\psi_-} \dfrac{\sin\tilde{\theta}}{\hat{\kappa}(\tilde{\theta})} \; d\tilde{\theta} \right) : -\psi_+ \le \theta \le \psi_- \right\}. \]
First, we prove $\alpha=0$. 
Denote $\hat{S}(\theta)$ by the support function of $\hat{\gamma}$, namely, 
\[ \hat{S}(\theta) := \sin \theta \int_{\theta}^{\psi_-} \dfrac{\cos\tilde{\theta}}{\hat{\kappa}(\tilde{\theta})} \; d\tilde{\theta} - \cos \theta \int_{\theta}^{\psi_-} \dfrac{\sin\tilde{\theta}}{\hat{\kappa}(\tilde{\theta})} \; d\tilde{\theta} \quad \mbox{for} \quad -\psi_+ \le \theta \le \psi_-. \]
By a simple calculation, we obtain 
\begin{align}
&\hat{S}_{\theta} (\theta) = \cos \theta \int_{\theta}^{\psi_-} \dfrac{\cos \tilde{\theta}}{\hat{\kappa}(\tilde{\theta})} \; d\tilde{\theta} + \sin \theta \int_{\theta}^{\psi_-} \dfrac{\sin\tilde{\theta}}{\hat{\kappa}(\tilde{\theta})} \; d\tilde{\theta}, \label{S-first}\\
&\hat{S}_{\theta\theta} + \hat{S} = -1/\hat{\kappa}. \label{S-second}
\end{align}
By integrating \eqref{S-second} and using the boundary condition \eqref{as-lim-kappa-3}, we have 
\begin{equation}\label{lim-kappa-3}
\int_{-\psi_+}^{\psi_-} \hat{S}(\theta) \; d\theta = L(\hat{\kappa}) + \cos\psi_+ \int_{-\psi_+}^{\psi_-} \dfrac{\cos\theta}{\hat{\kappa}(\theta)} \; d\theta. 
\end{equation}
Multiplying \eqref{as-lim-kappa-1} by $\hat{S}/\hat{\kappa}^2$ and integrating it, we obtain 
\begin{equation}\label{lim-kappa-4}
\int_{-\psi_+}^{\psi_-} \dfrac{\alpha\hat{S}}{\hat{\kappa}} \; d\theta = \int_{-\psi_+}^{\psi_-} \left( \hat{\kappa}_{\theta \theta} + \hat{\kappa} + \frac{\psi_+ + \psi_-}{L(\hat{\kappa})}\right) \hat{S} \; d\theta. 
\end{equation}
Now, we calculate the both sides of \eqref{lim-kappa-4}. 
For the right hand side, applying the integration by parts, \eqref{as-lim-kappa-2}, \eqref{as-lim-kappa-3}, \eqref{S-first}, \eqref{S-second} and \eqref{lim-kappa-3}, we obtain 
\begin{equation}\label{lim-kappa-5}
\begin{aligned}
&\; \int_{-\psi_+}^{\psi_-} \left( \hat{\kappa}_{\theta \theta} + \hat{\kappa} + \frac{\psi_+ + \psi_-}{L(\hat{\kappa})}\right) \hat{S} \; d\theta \\
= &\; \int_{-\psi_+}^{\psi_-} \hat{\kappa} ( \hat{S}_{\theta\theta} + \hat{S}) \; d\theta + \dfrac{\psi_+ + \psi_-}{L(\hat{\kappa})} \int_{-\psi_+}^{\psi_-} \hat{S}\; d\theta + [\hat{\kappa}_\theta \hat{S} - \hat{\kappa} \hat{S}_\theta]_{-\psi_+}^{\psi_-} = 0.
\end{aligned}
\end{equation}
Combining \eqref{lim-kappa-4} and \eqref{lim-kappa-5}, we have 
\[ \alpha \left(-\int_{-\psi_+}^{\psi_-} \dfrac{\sin\theta}{\hat{\kappa}(\theta)} \int_{\theta}^{\psi_-} \dfrac{\cos\tilde{\theta}}{\hat{\kappa}(\tilde{\theta})} \; d\tilde{\theta}d\theta  + \int_{-\psi_+}^{\psi_-} \dfrac{\cos \theta}{\hat{\kappa} (\theta)} \int_{\theta}^{\psi_-} \dfrac{\sin \tilde{\theta}}{\hat{\kappa}(\tilde{\theta})} \; d\tilde{\theta} d\theta\right) = \int_{-\psi_+}^{\psi_-} \dfrac{\alpha \hat{S}}{\hat{\kappa}} \; d\theta = 0. \]
By virtue of \eqref{as-lim-kappa-4}, it implies $\alpha = 0$. 

Next, we prove the last statement of Lemma \ref{lem:lim-kappa}. 
Since $\alpha$ is zero and $\hat{\kappa}$ satisfies \eqref{as-lim-kappa-1}--\eqref{as-lim-kappa-3}, $\hat{\kappa}$ is a stationary solution to \eqref{eq-theta}. 
Therefore, the moving curve constructed by \eqref{left-theta} and \eqref{posi-theta} with $\kappa(\cdot,t) \equiv \hat{\kappa}$ and arbitrary $x(-1,0)$ is a traveling wave for \eqref{apcf}--\eqref{bc2}. 
By virtue of the uniqueness of the traveling wave in Corollary \ref{cor:ex-tw}, the plane curve $\hat{\gamma}$ coincides with the profile curve $\mathcal{W}(0)$ obtained in Corollary \ref{cor:ex-tw} with \eqref{area-lim-kappa} up to the translation parallel to the $x$-axis. 
Therefore, the curvature $\hat{\kappa}$ of $\hat{\gamma}$ also coincides with the curvature $\kappa_{\mathcal{W}}$ of $\mathcal{W}(0)$. 
\end{proof}

Finally, we prove Theorem \ref{thm:main2}

\begin{proof}[Proof of Theorem \ref{thm:main2}]
Fix a positive constant $\varepsilon > 0$. 
By virtue of assumptions {\rm (A6)}, {\rm (A7)}, the higher regularity as in Lemma \ref{short-v}, \eqref{uni-nega-k} and $\frac{\partial}{\partial \theta} = \frac{1}{\kappa} \cdot \frac{\partial}{\partial s}$, we obtain the uniformly boundedness of $\|\kappa(\theta,t)\|_{C^k([-\psi_+, \psi_-])}$ with respect to $t \in [\varepsilon,\infty)$ for arbitrary $k \in \mathbb{N}$. 
We note that $\|\kappa_t(\theta,t)\|_{C^{k-2}([-\psi_+, \psi_-])}$ is also uniformly bounded by virtue of the differential equation in \eqref{eq-theta}. 
Applying the Ascoli-Arzel\'a theorem, we may see that there exists a sequence $t_i$ diverging to infinity and functions $\hat{\kappa}, \overline{\kappa} \in C^\infty([-\psi_+,\psi_-])$ such that 
\begin{equation}\label{main-1}
\kappa(\cdot,t_i) \to \hat{\kappa} \quad \mbox{in} \quad C^\infty([-\psi_+,\psi_-]), \quad \kappa_t(\cdot,t_i) \to \overline{\kappa} \quad \mbox{in} \quad C^\infty([-\psi_+, \psi_-]) 
\end{equation}
as $i \to \infty$. 
Hereafter, we prove that $\hat{\kappa}$ satisfies the assumptions in Lemma \ref{lem:lim-kappa}. 
The negativity of $\hat{\kappa}$ follows from \eqref{uni-nega-k}. 
By virtue of the representation of the length \eqref{leng-theta}, $L(t_i)$ converges to $L(\hat{\kappa})$, where $L(\hat{\kappa}) = -\int_{-\psi_+}^{\psi_-} \frac{d\theta}{\hat{\kappa}}$. 
Therefore, \eqref{as-lim-kappa-2} can be obtained by taking the limit of the boundary condition of \eqref{eq-theta}. 
The condition \eqref{as-lim-kappa-3} follows from the limit of \eqref{bc1-theta}. 
Since the area $A(t)$ represented by \eqref{area-theta} is preserved with respect time and the area $A(0)$ at initial time is positive, the limit of \eqref{area-theta} implies \eqref{as-lim-kappa-4}. 
In order to see that $\hat{\kappa}$ satisfies \eqref{as-lim-kappa-1}, it is enough to prove $\overline{\kappa} = \alpha \hat{\kappa}$ for some $\alpha \in \mathbb{R}$ by virtue of the differential equation in \eqref{eq-theta}. 

Let $\tilde{F}(t)$ be the functional defined by \eqref{def-lia}. 
We will apply Lemma \ref{lem:liapunov} to analyze the relation between the limit of the curvatures $\hat{\kappa}$ and $\overline{\kappa}$. 
We note that 
\begin{equation}\label{main-2} 
\int_{-\psi_+}^{\psi_-} \dfrac{\kappa_t^2}{\kappa^2} \; d\theta - \dfrac{1}{\psi_++\psi_-} \left(\int_{-\psi_+}^{\psi_-} \dfrac{\kappa_t}{\kappa} \; d\theta \right)^2
\end{equation}
is non-negative as we mentioned for \eqref{holder-ine}. 
By virtue of \eqref{iso-ine}, \eqref{uni-nega-k} and the uniformly boundedness of $\|\kappa(\cdot,t)\|_{C^1([-\psi_+,\psi_-])}$ with respect to $t \in [\varepsilon,\infty)$, there exist constants $M_3, M_4>0$ such that 
\begin{align}
&\dfrac{d}{dt} \tilde{F} (t) \ge M_3 \left(\int_{-\psi_+}^{\psi_-} \dfrac{\kappa_t^2}{\kappa^2} \; d\theta - \dfrac{1}{\psi_++\psi_-} \left(\int_{-\psi_+}^{\psi_-} \dfrac{\kappa_t}{\kappa} \; d\theta \right)^2\right), \label{main-3}\\
&\tilde{F} (t) \ge M_4 \label{main-4}
\end{align}
for $t \ge \varepsilon$. 
Integrating \eqref{main-3} on $[\varepsilon,\infty)$ and applying \eqref{main-4}, we have 
\[ \int_\varepsilon^\infty \left(\int_{-\psi_+}^{\psi_-} \dfrac{\kappa_t^2}{\kappa^2} \; d\theta - \dfrac{1}{\psi_++\psi_-} \left(\int_{-\psi_+}^{\psi_-} \dfrac{\kappa_t}{\kappa} \; d\theta \right)^2\right) \; dt \le \dfrac{M_2 - \tilde{F}(\varepsilon)}{M_1}. \]
Therefore, since \eqref{main-2} is non-negative, \eqref{main-2} converges to $0$ as $t \to \infty$. 
Thus, we obtain 
\[ \int_{-\psi_+}^{\psi_-} \dfrac{\overline{\kappa}^2}{\hat{\kappa}^2} \; d\theta = \dfrac{1}{\psi_++\psi_-} \left(\int_{-\psi_+}^{\psi_-} \dfrac{\overline{\kappa}}{\hat{\kappa}} \; d\theta \right)^2 \]
by virtue of the convergence \eqref{main-1} and the divergence of $t_i$ to infinity as $i \to \infty$. 
From the condition so that H\"{o}leder's inequality becomes an equality, we may see that $\overline{\kappa} = \alpha \hat{\kappa}$ for some $\alpha \in \mathbb{R}$. 

Therefore, we can apply Lemma \ref{lem:lim-kappa} and hence $\hat{\kappa}$ coincides with $\kappa_{\mathcal{W}}$ obtained in Corollary \ref{cor:ex-tw} with \eqref{area-lim-kappa}. 
From the convergence \eqref{main-1}, $\sup_{\theta \in [-\psi_+, \psi_-]} |\kappa(\theta,t_i) - \kappa_{\mathcal{W}}(\theta)|$ is sufficiently small for large $i \in \mathbb{N}$. 
Thus, Theorem \ref{thm:local-sta} can be applied to obtain the conclusion by replacing the initial time by $t_i$. 
\end{proof}

\end{document}